\documentclass[dvips, bernouilli]{imsart}

\usepackage{amssymb, amsthm, amstext, amsxtra, amsmath, amsfonts, amscd} 

\RequirePackage[OT1]{fontenc}
\RequirePackage[numbers]{natbib}
\RequirePackage[colorlinks,citecolor=blue,urlcolor=blue]{hyperref}

\startlocaldefs

\def\tends{\rightarrow}
\def\to{\rightarrow}
\def\E{\mathbb{E}}
\def\PP{\mathbb{P}}

\def\N{{\cal{N}}}

\DeclareMathOperator{\sign}{sign}
\DeclareMathOperator{\supp}{supp}

\DeclareMathOperator{\Var}{Var}

\newtheorem{theorem}{Theorem}         
\newtheorem{lemma}{Lemma}             
\newtheorem{definition}{Definition}[section]   
\newtheorem{corollary}{Corollary}     
\newtheorem{remark}{Remark}
\newtheorem{example}{Example}[section]
\endlocaldefs

\begin{document}

\begin{frontmatter}

\title{Besov regularity of functions with sparse random wavelet coefficients}
\runtitle{Besov regularity of stochastic wavelet expansions}


\author{\fnms{Natalia} \snm{Bochkina}\ead[label=e1]{N.Bochkina@ed.ac.uk}}
\address{\printead{e1}}

\affiliation{University of Edinburgh and Maxwell Institute, UK}

\address{School of Mathematics, University of
Edinburgh, Edinburgh EH9 3JZ, UK.
}

\runauthor{N.A. Bochkina }

\begin{abstract}
This paper addresses the problem of regularity properties of functions represented as an expansion in a wavelet basis with random coefficients in terms of  finiteness of their Besov norm with probability 1. Such representations are used to specify a prior measure in Bayesian nonparametric wavelet regression. Investigating regularity of such functions is an important problem since the support of the posterior measure does not include functions that are not in the support of the prior measure, and therefore determines the functions that are possible to estimate using specified Bayesian model. We consider a more general parametrisation than has been studied previously which allows to study a priori regularity of functions under a wider class of priors. We also emphasise the difference between the abstract stochastic expansions that have been studied in the literature, and the expansions actually arising in nonparametric regression, and show that the latter cover a wider class of functions than the former. We also extend these results to stochastic expansions in an overcomplete wavelet dictionary.
\end{abstract}


\end{frontmatter}


\section{Introduction}

This study is motivated by Bayesian wavelet estimators of the unknown function $f$ in the nonparametric regression model
\begin{equation}\label{eq:regression}
Y_i = f(i/n) + \varepsilon_i,\quad i=1,\ldots, n,
\end{equation}
where $\varepsilon_i$ are independent random variables with zero mean.
The key property that makes wavelet models a valuable tool is that the unknown function
has a sparse wavelet representation, i.e. $f$  is well approximated by
a function with a relatively small proportion of nonzero wavelet coefficients.  Within a Bayesian context,
the notion of sparsity is naturally modeled by a prior distribution for
the wavelet coefficients of $f$ that includes a point mass at $0$ and some other, often a continuous, distribution to model non-zero values of wavelet coefficients.
Such Bayesian models have been widely used to derive appropriate wavelet estimators using posterior mean, median, mode and Bayes factor estimators
\citep[e.g.][]{cly-geo-99, p:abr-sap-silv-98, p:vid-98} 
that perform well in practice and achieve  minimax rate of convergence, both under pointwise and $L^p$ losses, e.g. \citet{p:pensky-2003},  \citet{p:pensky-sap-07},  \citet{p:abr-point-opt-07}, \citet{p:boch-sap-09}, \citet{p:joh-sil-05}.
 For  a survey of Bayesian methods in wavelet nonparametric regression, see \citet{b:vid-survey}.

An important question that arises in selecting a prior distribution in the wavelet sequence space (i.e. in the space of wavelet coefficients) is to understand the corresponding a priori assumption on the function $f$. This question was first raised by \citet{p:abr-sap-silv-98}  who studied it for an independent mixture prior of the  atom at zero and a Gaussian distribution, and derived a necessary and sufficient condition for function $f$ to belong to a Besov space $B_{p,q}^s$ with probability 1. However, this type of prior distribution results in a suboptimal rate of convergence, since, as \citet{p:joh-sil-05} shows, the tail of the prior distribution should not be heavier than exponential to achieve the optimal convergence rate in $L^p$ norm.
 \citet{p:pensky-2003}, whilst studying optimality of
Bayesian wavelet estimators with normal and heavy tailed mixture
prior distributions over the Besov spaces, provided a {\it necessary}
condition for a function to belong to Besov spaces a priori with
probability 1.

Thus, the open question remains to identify the necessary and sufficient condition for a random function whose wavelet coefficients have the mixture distribution with non-Gaussian tails and for the cases $p=\infty$ or $q=\infty$ (that includes H\"older spaces), to belong to a Besov space almost surely. In practice it is important from two perspectives: visualisation of the functions that are generated by any chosen prior via simulation, and secondly, to specify a prior distribution that will include functions with required smoothness. This is a crucial property in proving consistency of Bayesian  wavelet estimators since it is impossible to achieve consistency over the Besov space that is not in the support of the prior distribution.
Bayesian wavelet estimators are also an important tool in constructing optimal estimators (in the minimax sense) in other nonparametric problems (e.g. \citet{Ray_Inv} in inverse problems  and  \citet{p:NicklGine2011} in density estimation), so the choice of the appropriate prior distribution for a function of given smoothness can provide a handle on creating an appropriate Bayesian estimator in many nonparametric problems.

The problem that has been studied by the authors quoted above is formulated as a prior distribution on all wavelet coefficients. However, in nonparametric wavelet regression the prior distribution  is specified  for the first $n$ wavelet coefficients at most, and asymptotic properties of Bayesian inference are studied in the limit $n\to\infty$. We also study a prior regularity properties of functions under this prior measure that is relevant to the support of the posterior distribution in nonparametric regression problem. In particular, we show that given the same hyperparameters, this prior measure includes a wider class of functions than the previously studied infinite dimensional prior model.

A similar question about the characterisation of the regularity properties of $f$ and a probability model for its continuous wavelet transform.
Continuous wavelet transform provides a greater flexibility of modelling functions of interest since it does not have dyadic restrictions on the indices of the
wavelet function. Since the wavelet coefficients are related to the size of the discontinuity of a function at the corresponding frequency and location, continuous wavelet transform is particularly well suited to practical problems where the frequencies of the non-zero wavelet coefficients are of interest,
and it finds its application to data in meteorology, oceanography and medicine, e.g. \citet{p:CWT-Ouergli}, \citet{p:CWT-Polygiannakis}, \citet{p:CWT-Mochimaru}. \citet{p:abr-sap-silv-00} derived a necessary and sufficient condition for a function with a random sample of continuous wavelet coefficients whose distribution has Gaussian distribution, to belong to a Besov space with probability 1. We derive such a condition for a model with more general assumptions on the distribution of the wavelet coefficients.

Using Besov spaces to characterise the regularity properties of the function is most
 natural since the characterisations are expressed by simple conditions on wavelet coefficients. The Besov spaces
 consist of functions of different levels of regularity and include, in
particular, the well-known Sobolev  $W_2^s = B_{2,2}^s$ and
H\"older $\mathbb{C}^s = B_{\infty,\infty}^s$ spaces of smooth
functions, and also less traditional spaces, like the space of
functions of bounded variation, sandwiched between $B_{1,1}^1$ and
$B_{1,\infty}^1$ and spaces of spatially inhomogeneous functions $B_{pq}^s$ with $p\in [1,2)$.
 For more details see  \citet{b:mey-92} and
 \citet{b:tri-83}.   Besov spaces and Besov balls have been widely used in wavelet
context \citep[e.g.][]{p:don-joh-94a, p:don-joh-94b, p:don-joh-95, p:abr-sap-silv-98, p:abr-sap-silv-00}  
due to containing spatially non-homogeneous functions and functions of finite total variation which can be represented in
terms of wavelets. Regularity of the solutions  of some stochastic PDEs is studied in terms of Besov norm by using their wavelet representation \citep{p:BesovStoch, p:BesovPDE, b:RegStochasticPDE}. \citet{p:don-joh-95} discuss the relevance of Besov spaces for various scientific
problems.

The main aim of this paper is to derive is to derive necessary and sufficient conditions on the parameters of the prior distribution for a given distribution of the nonzero wavelet coefficients under the orthogonal wavelet transform, without assuming any particular form of the parameters. We derive these conditions in two settings: a) the prior distribution is specified for all wavelet coefficients (infinite dimensional case) that is relevant for infinite dimensional problems, such as studying regularity of the solutions of stochastic PDEs, and b) the prior distribution is specified in the same way as in the nonparametric regression, i.e. the mixture prior is specified only for a finite number of wavelet coefficients that depends on the sample size $n$, and the regularity is studied as the sample size tends to infinity. It turns out that a modification of the proof given by \citet{p:abr-sap-silv-00} allows to apply the derived results to the case of the continuous wavelet transform, so we state these conditions here.

The paper is organised as follows. The definitions of the orthonormal and continuous wavelet transforms and of the Besov spaces are given in Section~\ref{sec:DefWTandBS}. In Section~\ref{sec:regularDWT} we extend the result of  \citet{p:abr-sap-silv-98} -- using the same form of the parameters -- to the case of  a more general distribution of the nonzero wavelet coefficients, and the main result for arbitrary form of the parameters and general distribution is given in Section~\ref{seq:General}. In Section~\ref{sec:regression} we state such a result for the prior distributions used in nonparametric regression.  
And finally, in Section~\ref{sec:CWT} we state  necessary and sufficient conditions for a stochastic model for a continuous wavelet transform without assuming any particular form of the continuous analogues of the parameters defining the model. The proofs are given in the appendix. As a part of the proof, we derive a strong law of large numbers for a sequence of random variables of random length (Lemma~\ref{lem:LLN} in the appendix).

The results derived here under the nonparametric regression model can also be used to study a posteriori regularity properties of estimators of function $f$ where the hyperparameters depend on the observed data, since the posterior distribution of wavelet coefficients for fixed hyperparameters is of the same form as the prior, with a different distribution $H$ which depends on the data. This question is a subject of a separate investigation.

\section{Wavelet transform and Besov spaces}\label{sec:DefWTandBS}


\subsection{Orthogonal wavelet transform}\label{sec:wavelet}

A wavelet basis is determined by a wavelet and a scaling function
$\psi(x)$ and $\phi(x)$ which we assume to be of regularity $r$
and have periodic boundary conditions on $[0, 1]$. A set of
functions $\{\phi_{j_0,m}(x), \psi_{jk}(x), j \geqslant j_0\geqslant 0,
k=0,\dots,2^j-1,
m=0,\dots,2^{j_0}-1\}$ forms an orthonormal basis of $L^2([0,1])$ for any $j_0\in \mathbb{Z}_+$
where the functions $\psi_{jk}(x)$ are derived from the wavelet
function by dilation and translation:
$\psi_{jk}(x)=2^{j/2}\psi(2^{j} x - k)$.

Since wavelets form an orthogonal basis of $L^2([0,1])$, any $f\in L^2([0,1])$ can be written as
\begin{equation}\label{eq:WaveletFn}
f(x) = \sum_{m=0}^{2^{j_0}-1} u_{j_0,m}\phi_{j_0,m}(x) + \sum_{j=j_0}^{\infty} w_{jk} \psi_{jk}(x),
\end{equation}
where the wavelet and the scaling coefficients are given by
\begin{eqnarray}\label{eq:WCdef}
w_{jk} &=& \int_{0}^1 f(x) \psi_{jk}(x) dx, \quad j\geqslant j_0,
\quad k=0,\dots, 2^j-1,\\
u_{j_0,m} &=& \int_{0}^1 f(x) \phi_{jk}(x) dx, \quad m=0,\dots, 2^{j_0}-1
\end{eqnarray}
due to orthogonality of wavelet basis. Examples of wavelet functions can be found in \citet{b:vid-99}.


\subsection{Probability model based on orthonormal wavelet transform}

Now we consider a random function $f$ with an orthonormal wavelet decomposition~\eqref{eq:WaveletFn} where the wavelet coefficients $w_{jk}$ are random and have the following density with respect to the Lebesgue measure:
\begin{equation}\label{eq:prior}
w_{jk} \sim p_j(w) = (1-\pi_j) \delta_0(w) + \pi_j \tau_j h(\tau_j w), \quad \text{independently}
\end{equation}
where $\delta_0(w)$ is Dirac delta function, $\pi_j \in [0,1]$, $\tau_j>0$  and distribution with density $h$ with respect to the Lebesque measure has cumulative distribution function $H(w)$ which is continuous at $w=0$
for identifiability. Note that we do not make assumptions of
symmetry or continuity of distribution $H$  (except at $0$).


Before addressing the problem of regularity of random function $f$, we study the number of its non-zero wavelet coefficients.


Denote the number of nonzero wavelet coefficients at level $j$ by
${\cal{N}}_j$, and the number of all nonzero wavelet coefficients
by ${\cal{N}} = \sum_{j=j_0}^{\infty} {\cal{N}}_j$. To study whether $\N_j$ and $\N$ are finite, we apply an argument similar to that of \citet{p:abr-sap-silv-98}. It follows from the
probabilistic model (\ref{eq:prior}) that ${\cal{N}}_j$ has
binomial distribution with parameters $2^j$ and $\pi_j$. Therefore
the expected number of non-zero wavelet coefficients at level $j$
of wavelet decomposition is $2^j\pi_j$.

In case $\sum_{j=j_0}^{\infty } 2^j\pi_j < \infty$, number
$\cal{N}$ of the non-zero wavelet coefficients at all
decomposition levels is finite almost surely because its
distribution is proper (this can be shown, for instance, using the
characteristic function of distribution of $\cal{N}$).
In terms of regularity properties, this implies that  a function
with such wavelet coefficients almost surely belongs to the same
Besov spaces as the wavelet function, i.e. to the Besov spaces
with parameters $0<s<r$, $1\leqslant p, q \leqslant \infty$ where
$r$ is the regularity of the wavelet function, since $f$ is a
finite linear combination of wavelet functions with probability 1.

If $\sum_{j=j_0}^{\infty } 2^j\pi_j = \infty$, we consider two separate cases: a) $2^j\pi_j\tends const$, the number of non-zero wavelet coefficients at high resolution levels $j$ asymptotically has the same expectation (and the same
distribution), and b) $2^j\pi_j \tends \infty$ as $j\tends\infty$,
i.e. the expected value of $\N_j$ increases to infinity.


\subsection{Besov spaces}\label{sec:BesovDef}

 Besov spaces $B_{p,q}^s$ of functions can be characterised by the Besov sequence norm $b^s_{p,q}$ of the wavelet
coefficients of its elements for $0<s<r$ where $r$ is regularity of wavelet function \citep{p:don-joh-94a, p:don-joh-94b, p:don-joh-95,
p:don-joh-98}. Explicit definition of Besov spaces can be found in \citet{b:tri-83}.

Define Besov sequence norm $b^s_{p,q}$ on
wavelet coefficients for $p, q \geqslant 1$, $s>0$ and $s' = s
+\frac 1 2 - 1/p$, as follows:
\begin{eqnarray}\label{eq:normequ}
||\mathbf{w}||_{b^s_{p,q}} &=& ||\mathbf{u}_{j_0}||_p + \left\{
\sum_{j=j_0}^{\infty} 2^{j s' q} ||\mathbf{w}_{j}||_p^q
\right\}^{1/q},
\quad 1 \leqslant q < \infty,\\
||\mathbf{w}||_{b^s_{p,\infty}} &=& ||\mathbf{u}_{j_0}||_p+ \sup_{j
\geqslant j_0} \left\{ 2^{js'} ||\mathbf{w}_{j}||_p \right\}, \quad q
= \infty,
\end{eqnarray}
where $\mathbf{u}_{j_0}= (u_{j_0, 0}, u_{j_0, 1}, \dots, u_{j_0,
2^{j_0}-1} )$ is a vector of scaling coefficients at level $j_0$,
vectors $\mathbf{w}_j=(w_{j, 0}, w_{j, 1}, \dots, w_{j, 2^j-1} )$
consist of wavelet coefficients at level $j$ for $j \geqslant
j_0$, and the vector $\mathbf{w}=(\mathbf{u}_{j_0}, \mathbf{w}_{j_0},
\mathbf{w}_{j_0+1},\dots)$ is the union of these vectors, i.e. the
complete set of wavelet coefficients.

The key property of this
norm defined on wavelet coefficients is that it is equivalent to
the Besov norm $B_{p,q}^s [0, 1]$ of the corresponding function
provided that the regularity of the wavelet function $r$ is such
that $r > s > 0$ \citep[Theorem 2]{p:don-joh-98}.

\section{Orthogonal wavelet transform with simple parametrisation}\label{sec:regularDWT}

In this section we connect the probabilistic model \eqref{eq:prior} for wavelet
coefficients of the orthogonal wavelet transform
with regularity properties of the function with such coefficients,
expressed in terms of the Besov norm defined in Section
\ref{sec:BesovDef}. Since the coefficients are random, these results
hold with probability 1.

\subsection{Assumptions on the prior}\label{sec:AssumptionH}

Consider the probabilistic model \eqref{eq:prior} for wavelet coefficients with the following parametrisation considered by \citet{p:abr-sap-silv-98}:
\begin{equation}\label{eq:hyper0}
\tau_j^2 = 2^{-\alpha j}C_1, \quad \pi_j=\min(1, 2^{-\beta j}C_2),
\end{equation}
where $C_1$ and $C_2$ are positive constants possibly dependent on
sample size $n$, $\alpha \geqslant 0$, $\beta \geqslant 0$, $\alpha+\beta >0$.
Such choice is motivated by the sparsity property of wavelet transform, i.e. that in most
cases a function can be well approximated by a finite number of
non-zero wavelet coefficients, therefore parameters $\tau_j^2$ and
$\pi_j$ are chosen to tend to zero as the decomposition level $j$
tends to infinity.
To visualise the effect of $\alpha$ and $\beta$ on the
regularity of the function, see simulation study of \citet{p:abr-sap-silv-98} for different values of $\alpha$ and
$\beta$ with $H$ being normal distribution.

In case $p=\infty$, the norm $||w_j||_p$ is the maximum of the absolute values of the wavelet coefficients at level $j$, we need to specify tail behaviour of $H$. We consider two types of distributions that belong to the domain of the attraction of two classes of ``max-stable'' distributions: $e^{-e^{-x}}$ and $e^{-x^{-\ell}}$, $x>0$, $\ell >0$ \citep[see][]{b:led-83}. The two considered types of attraction apply in the case the upper point of distribution of $|w_{jk}|$ is infinity which is typically the case in case of prior distribution on wavelet coefficients.

\begin{definition}\label{def:EVT} 1. Distribution with cumulative distribution function $F$ belongs to the domain of attraction $D(e^{-e^{-x}})$ if there exists function $g(x)$ such that  for any $x\in \mathbb{R}$
$$
\lim_{t\to +\infty} \frac{1 - F(t+xg(t))}{1 - F(t)} = e^{-x}.
$$

2. Distribution with cumulative distribution function $F$ belongs to the domain of attraction $D(e^{-x^{-\ell}})$, $\ell >0$, if for any $x>0$
$$
\lim_{t\to +\infty} \frac{1 - F(tx)}{1 - F(t)} = x^{-\ell}.
$$
\end{definition}
In the first case, one can take $g(x) = (1 - F(x))/F'(x)$. For example, normal distribution $N(0, \sigma^2)$ belongs to the domain of attraction $D(e^{-e^{-x}})$, as well as a distribution with cumulative distribution $F(x) = 1 - e^{-|x|^m}$ ($m>0$) which has $g(x) =  x^{-(m-1)}/m$. On the other hand, distributions with a polynomial tail, such as Pareto or $t$ distribution, belong to the domain of attraction $D(e^{-x^{-\ell}})$.

Fix some values $1 \leqslant p,q \leqslant\infty$ (parameters of Besov space $B_{pq}^s$). Then the key assumptions about
the distribution with cumulative distribution function (c.d.f.) $H$ are given below.

\par
Assumption H:
{\it
Suppose that random variable $\xi$ has c.d.f. $H$.
\begin{enumerate}
\item $0\leqslant \beta < 1$, $1 \leqslant p < \infty$, $1 \leqslant q
\leqslant \infty$: assume that $\E |\xi|^p < \infty$.
\item { $0\leqslant \beta < 1$, $p = \infty$, $1 \leqslant q
\leqslant \infty$: assume that $H_+(x) \stackrel{def}{=} H(x)-H(-x)$ is of
one of the following types:}
\begin{enumerate}
\item\label{a} {  $H_+(x) \in D(e^{-x^{-\ell}})$, $\ell>0$; if $q<\infty$, assume that
$\ell>q$; }
\item\label{b} { $1-H_+(x)\leq c_m e^{-(\lambda
x)^m} $ for large enough $x>0$, and $m, \lambda, c_m>0$.}
\end{enumerate}
\item $\beta = 1$, $1 \leqslant p \leqslant \infty$, $1 \leqslant q < \infty$:  assume that
$\E |\xi|^q < \infty$.
\item $\beta = 1$, $1 \leqslant p \leqslant \infty$, $q = \infty$:
assume that $\exists \epsilon>0$ such that $$\E [\log(|\xi|)
I(|\xi|>\epsilon)]<\infty.$$

\end{enumerate}
}

Note that in cases 1, 3, 4 it suffices to know the finiteness of
moments of distribution $H$, i.e. it must have a finite absolute
moment of some order greater or equal to $1$. But in case 2
($p=\infty$, $0\leqslant \beta <1$) it is necessary to know the tail behaviour
of the distribution. This is due
to different asymptotic distribution of the maximum of a large
number of independent identically distributed random variables
with different tail behaviour (see proof of Theorem \ref{th:mainBesov} in
Appendix~\ref{seq:ProofDWT}). Also, we can see that cases
$0\leqslant \beta <1$ and $\beta=1$ that correspond to equal and
increasing to infinity expected number of nonzero wavelet coefficients at different decomposition levels, require
different assumptions on distribution $H$.

\subsection{Regularity properties}\label{sec:mainBesov}

Now we formulate the necessary and sufficient conditions on the hyperparameters $\alpha$ and $\beta$ for function $f$ to belong to a Besov space.

To distinguish between cases of distributions with polynomial and exponential tails stated in Assumption H in Section~\ref{sec:AssumptionH}, we
introduce an auxiliary variable:
\begin{displaymath}\label{eq:defDeltaH}
\delta_H = \left\{ \begin{array}{ll} \frac{1-\beta}{l}, & \quad H
\quad \text{has polynomial tail and }
p=\infty, \\
0, & \quad \text{otherwise}.
\end{array} \right.
\end{displaymath}

Now we can formulate the criterion.

\begin{theorem}\label{th:mainBesov}
Suppose that $\psi$ and $\phi$ are wavelet and scaling functions of regularity $r$. 
Consider function $f$ and its wavelet transform \eqref{eq:WaveletFn} under prior \eqref{eq:prior} satisfying Assumption H under parametrisation \eqref{eq:hyper0}.

Then, for any fixed value of $||u_{j_0}||_p$ and for any $s\in (0,r)$, $f \in B^s_{p,q}$ almost surely if and only if
\begin{eqnarray*}
 s  &\leq&   \frac{\alpha-1} 2 +\frac{\beta} p, \quad \quad \quad \text{if } \,\, 0\leqslant
\beta < 1, \, p<\infty,  \, q=\infty,\\
 s &<&   \frac{\alpha-1} 2 +\frac{\beta} p -\delta_H,\quad \text{otherwise}\\
\end{eqnarray*}

\end{theorem}
This theorem follows from Theorem~\ref{th:RegOWT}.

\begin{table}
\begin{center}
\begin{tabular}{|l| l| l| }
\hline
Parameters & $0\leqslant \beta<1$ & $\beta=1$ \\
\hline
$1\leqslant p <\infty$, $1\leqslant q < \infty$ & $s < (\alpha-1)/2+\beta/p $ & $s <(\alpha-1)/2+1/p$ \\
\hline
$1\leqslant p <\infty$, $q = \infty$ & $s \leqslant (\alpha-1)/2+\beta/p$ & $s <(\alpha-1)/2+1/p${$^*$}\\
\hline $p=\infty$, $1\leqslant q <l $ or $q=\infty$, & & \\
$1 - H(x) +H(-x) \sim c_l x^{-l}$ & $s <(\alpha-1)/2+ (\beta-1) /\ell $ &
$s <(\alpha-1)/2$ \\
\hline
$p=\infty$, $1\leqslant q \leqslant \infty$ & &\\
$1 - H(x) +H(-x) \sim c_m e^{-(\lambda x)^m}$ &$s <(\alpha-1)/2 $ & $s <(\alpha-1)/2 $ \\
\hline
\end{tabular}
\end{center}
\caption{\textsl{The necessary and sufficient condition stated in Theorem \ref{th:mainBesov} for different values of parameters  (condition $^*$ is sufficient).} }
\end{table}

\begin{table}
\begin{center}
\begin{tabular}{|l| l| }
\hline
Parameters & The necessary and sufficient condition \\
\hline
$0\leqslant \beta<1$ , $1\leqslant p <\nu$, $1\leqslant q < \nu$ & $s < (\alpha-1)/2 +\beta/p $ \\
\hline
$0\leqslant \beta<1$ , $1\leqslant p <\nu$, $q = \infty$ & $s  \leqslant (\alpha-1)/2+\beta/p $ \\
\hline
$0\leqslant \beta<1$ , $p=\infty$, $1\leqslant q < \nu$ & $s < (\alpha-1)/2+ (\beta-1)/\nu$ \\
\hline
$\beta=1$, $1\leqslant p \leqslant \infty$, $1\leqslant q <\nu$ & $s < (\alpha-1)/2+ 1/p$ \\
\hline
$\beta=1$, $1\leqslant p \leqslant \infty$, $q=\infty$ & $s < (\alpha-1)/2 + 1/p ${$^*$}\\
\hline
\end{tabular}
\end{center}
\caption{\textsl{The necessary and sufficient condition in case $h=t_{\nu}$ (condition $^*$ is sufficient). } }
\label{table:besovT}
\end{table}

For finite $p$, the regularity parameter $s$ which is related to smoothness of $f$, must be compensated by the hyperparameters
of the variance $\alpha$ and of the proportion of non-zero wavelet
coefficients $\beta$ scaled by $p$.
If $p$ is infinite, the proportion of non-zero wavelet coefficients related to the hyperparameter $\beta$ ceases
to matter in the case of a distribution with quickly decreasing tail such as
power exponential. However, for a distribution with a more slowly decreasing tail such as polynomial
it is still essential:
$$
s  < \frac{\alpha-1} 2 -\frac{1-\beta} l.
$$
Therefore, for $p=\infty$, a heavier-tailed prior corresponds to a less regular function, with high probability.

\subsection{Examples}

In order to illustrate the results of the theorem, we consider several examples of
distribution $H$.

\begin{enumerate}
\item {\bf Normal distribution,} $h(x)=\frac 1 {\sqrt{2\pi}\sigma}
e^{-x^2/2\sigma^2}$. In this case, assumption H is satisfied for
all combinations of the parameters. Therefore the necessary and
sufficient condition for a function with wavelet coefficients
obeying the model (\ref{eq:prior}) with $H(x)=\Phi(x/\sigma)$ to belong
to $B^s_{p,q}$ with probability 1 is
\begin{eqnarray*}
s  &\leq& \frac{\alpha-1} 2 +\frac{\beta} p
 \quad\quad \text{if } \quad 1\leqslant p<\infty, \quad q=\infty, \quad 0\leqslant \beta<1,\\
 s   &<& \frac{\alpha-1} 2 +\frac{\beta} p, \quad otherwise.
\end{eqnarray*}
This result coincides with the one stated in  \citet{p:abr-sap-silv-98}.

\item {\bf Laplacian (double-exponential) distribution,} $h(x) = \frac {\lambda} 2 e^{-\lambda|x|}$.
Since Laplacian
distribution has exponential tail and all its moments are finite,
the necessary and sufficient condition is the same as for the
normal distribution.

\item {\bf T distribution with $\nu$ degrees of freedom,} $h(x) =
C_{\nu}(1+x^2/\nu)^{-\frac{\nu+1}{2}}$, $\nu\geqslant 1$. T
distribution with parameter $\nu\geqslant 1$ has finite moments of order less than $\nu$ and belongs to $D(e^{-x^{-\ell}})$ with $\ell = \nu$. The necessary and
sufficient condition for a function to belong to the Besov spaces
in terms of its wavelet coefficients in each case is given in Table
\ref{table:besovT}. Due to a finite number of finite absolute
moments, the conditions are more restrictive than for
distributions with power-exponential tail.

\end{enumerate}

\section{General case, orthogonal wavelet transform}\label{seq:General}

\subsection{Regularity properties}

In this section we state the necessary and sufficient condition in terms of arbitrary $\tau_j^2$ and $\pi_j$
for a function to belong to Besov space $B_{pq}^s$ with probability 1.

Below, we understand that for $p=\infty$, $1/p$ stands for 0.

\begin{theorem}\label{th:RegOWT} Consider a function $f$ with wavelet coefficients following
distribution (\ref{eq:prior}). Suppose that parameters of the
Besov spaces are restricted to $1 \leqslant p,q \leqslant \infty$,
$0 < s < r$, where $r$ is the regularity of the wavelet and
scaling functions. Denote $s' = s-1/p+1/2$ and $\xi \sim H$.

\begin{enumerate}
\item\label{it1:owt} $p < \infty$: assume   $\E |\xi|^{p} < \infty$,  $\pi_j 2^j\rightarrow\infty$ as $j\rightarrow\infty$ and $\pi_j 2^j$ is a monotonically increasing sequence  for large enough $j$.

Then, for any fixed value of $||u_{j_0}||_p < \infty$,
\begin{eqnarray*}
\PP\{ f \in B_{p q}^s \} =1 \quad \Leftrightarrow \quad
\sum_{j=j_0}^{\infty} \left[2^{j (s' +1/p)} \tau_j
\pi_j^{1/p}\right]^q &<& \infty, \quad  1 \leqslant q<\infty;\\
\sup_{j \geqslant j_0} [2^{j (s' +1/p)} \tau_j \pi_j^{1/p}] &<&
\infty, \quad q=\infty.
\end{eqnarray*}

\item\label{it2:owt} $p = \infty$. Assume that  $\pi_j 2^j > 1$
 is a sequence monotonically increasing to $\infty$ for large $j$.
Denote $H_{+}(x) = H(x)-H(-x)$, $x\geqslant 0$, and assume that
either $H_{+}(x)  \in D(e^{-e^{-x}})$ or $H_+(x) \in D(e^{-x^l})$,
for some $l>0$. Define $b_j$: $H_+(b_j) = 1 - (\pi_j 2^j)^{-1}$.

Assume that
\begin{enumerate}
\item   $ [ 2^{j} \pi_j b_j H_+'(b_j)]^{-1} \log j \to 0$ as $j
\to \infty$ if $H_{+}(x) \in D(e^{-e^{-x}})$;
\item  $q<\ell$ if $H_+(x) \in D(e^{-x^{-\ell}})$.
\end{enumerate}

Then, for any fixed value of $||u_{j_0}||_p < \infty$,
\begin{eqnarray*}
\PP\{ f \in B_{pq}^s \} =1 \quad \Leftrightarrow \quad
\sum_{j=j_0}^{\infty} [2^{j s' } \tau_j b_j]^{q} &<& \infty, \quad
1 \leqslant q< Q;\\  \quad \sup_{j \geqslant j_0} [2^{j s'} \tau_j
b_j] &<& \infty, \quad q=\infty,
\end{eqnarray*}
where  $Q = \infty$ if $H_{+}(x)  \in D(e^{-e^{-x}})$, and $Q = \ell$
if $H_+(x) \in D(e^{-x^{-\ell}})$.

\item\label{it3:owt} $1 \leqslant p \leqslant \infty$, $q <
\infty$, $\pi_j 2^j \to const$ as $j\to \infty$. Assume that $\E |\xi|^q < \infty$.  Then, for any fixed value of
$||u_{j_0}||_p < \infty$,
\begin{eqnarray*}
\PP\{ f \in B_{pq}^s \} =1 \quad \Leftrightarrow \quad
\sum_{j=j_0}^{\infty} [2^{j s'} \tau_j]^q < \infty.
\end{eqnarray*}

\item\label{it4:owt} $1 \leqslant p \leqslant \infty$, $q =
\infty$, $\pi_j 2^j \to const$ as $j\to \infty$.  Assume that $M(j) = \tau_j^{-1}
2^{-j s'}$ is a monotonically increasing sequence for $j\geq \tilde{j}$ for some $\tilde{j}\geq j_0$, and that $M(x)$ is extended to $[\tilde{j}, \infty)$ in such a way that function $M(x)$ (and hence $M^{-1}(x)$ is monotonically increasing.

Then, for any fixed
value of $||u_{j_0}||_p < \infty$,
$$\PP\{ f \in B_{pq}^s \} =1 \,\, \Leftrightarrow \,\, \exists \ c > 0: \,\, \E
M^{-1}(|\xi|/c) < \infty.$$

\item\label{it5:owt} $\sum_{j=j_0}^{\infty} \pi_j 2^j < \infty$. Assume that $\PP(|\xi|<\infty)=1$. 
Then, for any fixed value of $||u_{j_0}||_p < \infty$, $\PP\{ f \in
B_{pq}^s \} =1$ for  $1\leq p, q\leq \infty$ and $ s\in(0,r)$.

\end{enumerate}

\end{theorem}

Note that in case \ref{it4:owt}, the assumption and the necessary
and sufficient conditions are reversed, i.e. we make an assumption
on parameter $\tau_j$ and the necessary and sufficient condition
is given in terms of a finite moment of a function of $\xi$. For
example, if $\tau_j^2 = c_{\tau} 2^{-\alpha j}$, as considered in
the previous section, $M(j) = c_{\tau}^{-1} 2^{(\alpha/2 - s')
j}$, so the assumption is that $\alpha > 2 s'$ and the necessary
and sufficient condition is that $\exists c > 0$:  $\E \log_2
(|\xi| c_{\tau}/c) <\infty$. In case $\tau_j = c_{\tau} j^{\gamma}
2^{- 2 j s'}$, the assumption is that $\gamma < 0$ and that $\E
|\xi|^{-2/\gamma} < \infty$, i.e. the necessary and sufficient
condition becomes stronger for a slower decrease of $\tau_j$ for
large $j$.

For finite $p$ and $q$, sufficient conditions for $f\in B_{pq}^s$ almost surely
in terms of $\nu_j=1/\tau_j$ and odds $\beta_j =
\frac{1-\pi_j}{\pi_j}$ was given by  \citet{p:pensky-2003}.
Their conditions are the same or stronger than the assumptions and the necessary and sufficient conditions stated in Theorem~\ref{th:RegOWT}.
 If $p<q$ (part of cases \ref{it1:owt} and \ref{it2:owt}), their sufficient condition is stronger since we have the series with power $q$ of $2^{j(s +1/2)}\tau_j \pi_j^{1/p}$ instead of $\min(p,q)$ as  in \citet{p:pensky-2003}. In case \ref{it1:owt} and  $p>q$, our assumption $\E |\xi|^p<\infty$ is weaker than $\E |\xi|^{\max(p,q)}<\infty$ of \citet{p:pensky-2003}, and in case  \ref{it3:owt} it is weaker if $q<p$.  In case of a constant expected number of all nonzero wavelet coefficients, we do not need to assume finiteness of any absolute moment of distribution $H$, as long as it is not degenerate, i.e. if $\PP(|\xi|<\infty)=1$.

\begin{remark}
Only in the case $p =\infty$, $\pi_j 2^j >1$,
parameters of the tail behaviour of the distribution $H$ are
explicitly related to the condition on the parameters of the Besov
spaces.
\end{remark}

Next we discuss the asymptotic behaviour of $b_j$ defined  in the theorem that is used to derive Theorem~\ref{th:mainBesov}.

\begin{remark}
Consider $H(x)$ such that $1-H_+(x) = e^{-|x|^{m}} [1+o(1)]$ as $x\to \infty$, for some $m>0$ and $c_m>0$. This distribution belongs to the domain of attraction of $e^{-e^{-x}}$, with $b_j =  [\log (2^j\pi_j)]^{1/m}[1+o(1)]$.  Then, one of the assumptions on $H$ in case $p=\infty$ and $2^j\pi_j\to \infty$ is
$$
\frac{\log j}{ 2^{j} \pi_j b_j H_+'(b_j)}  = \frac{\log j}{ m \log (2^j\pi_j)[1+o(1)]} \to 0\,\,\text{as}\,\,  j\to \infty
$$
which is satisfied if $2^j\pi_j$ increases to infinity faster than any power of $j$.

If $1-H_+(x) \leq c_m e^{-|x|^{m}} $ for large enough $x>0$, for some $m>0$ and $c_m>0$, it belongs to the same domain of attraction, and
$$
(2^j\pi_j)^{-1} = 1-H_+(b_j) \leq c_m e^{-b_j^{m}} \,\, \Leftrightarrow \,\,   b_j \leq [\log(c_m 2^j\pi_j)]^{1/m}.
$$
\end{remark}

\begin{remark}
Consider $H(x)$ such that $1-H_+(x) =  |x|^{-\ell}[1+o(1)]$ as $x\to \infty$ for some $\ell>0$. It is easy to verify that it belongs to the domain of attraction of $e^{-x^{-\ell}}$, with $b_j = (2^j\pi_j)^{1/\ell}[1+o(1)]$ as $j\to \infty$ (case $p=\infty$ and $2^j\pi_j\to \infty$ as $j\to \infty$ of Theorem~\ref{th:RegOWT}).

Moreover, if $H_+(x)\in D(e^{-x^{-\ell}})$ for some $\ell>0$, using Definition~\ref{def:EVT},  as $j\to \infty$, for any $x>0$ we have
$$
1-H_+(xb_j) = x^{-\ell} [1-H_+(b_j)](1+o(1)).
$$
Taking some finite $x_0$ where $x_0>0$ is independent of $j$ and is such that  $1-H_+(x_0)\in(0,1)$, e.g. $1-H_+(x_0)=1/2$, we have with $x=x_0/b_j$:
$$
1-H_+(x_0) = x_0^{-\ell} b_j^{\ell}/(2^j\pi_j)(1+o(1)),
$$
i.e. $b_j = c_0(2^j\pi_j)^{1/\ell}(1+o(1))$ with a finite positive constant $c_0 = x_0^{\ell}/(1-H_+(x_0))$.
\end{remark}

\subsection{Parametrisation with three hyperparameters}

For the parametrisation considered in Section \ref{sec:regularDWT},
Theorem \ref{th:mainBesov} implies that functions with normal and
exponential tails of the distribution of their wavelet
coefficients fall in exactly the same Besov spaces almost surely. Also,
parameter $q$ of the Besov spaces is not strongly related to the
hyperparameters (\ref{eq:hyper0}) of probabilistic model. Thus, we may wish to reparametrise $\tau_j$ and $\pi_j$ by introducing a
third hyperparameter $\gamma$ which
is related to parameter $q$ and 
which separates the cases of normal and exponential distributions.
Following \citet{p:abr-sap-silv-98}, we keep parameter $\pi_j$ the same and introduce additional factor
to the parameter $\tau_j$:
\begin{equation}\label{eq:hyper1}
\tau_j^2 = j^{\gamma} 2^{-\alpha j}C_1, \quad j>0, \quad \tau_0^2
= C_1,
\end{equation}
where $\gamma$ takes values in $\mathbb{R}$. Define
\begin{displaymath}
\delta = \left\{ \begin{array}{ll}
s + \frac 1 2 -\alpha/2 -\beta/p, & \quad p<\infty, \\
s + \frac 1 2 -\alpha/2, & \quad p=\infty.
\end{array} \right.
\end{displaymath}
 Applying Theorem \ref{th:RegOWT}, we obtain the following statement.

\begin{corollary} Let $0\leqslant\beta<1$, $p=\infty$, $1\leqslant q \leqslant \infty$, and suppose that $\psi$ and $\phi$ are wavelet and scaling functions of regularity $r$, where $0 < s < r$. Consider the Bayesian model (\ref{eq:prior}) under parameterisation \eqref{eq:hyper1} for wavelet
coefficients $w_{jk}$  of the function $f$.

\begin{enumerate}
\item {\bf $H$ is normal distribution.} For any fixed value of
$u_{00}$, $f\in B_{\infty,q}^s $ if and only if
\begin{eqnarray*}
& 1\leqslant q<\infty: & \quad \text{either} \quad (\delta<0), \quad
\text{or} \quad (\delta=0\quad
\text{and} \quad \gamma< -2/q-1);\\
& q=\infty: & \quad \text{either} \quad (\delta<0), \quad \text{or}
\quad (\delta=0\quad \text{and} \quad \gamma\leqslant -1).
\end{eqnarray*}
This coincides with Theorem 2 stated in \citet{p:abr-sap-silv-98}.

\item {\bf $H$ is Laplacian distribution.} For any fixed value of
$u_{00}$, $f\in B_{\infty,q}^s$ if and only if
\begin{eqnarray*}
& 1\leqslant q<\infty: &\quad either \quad (\delta<0), \quad or
\quad
(\delta=0\quad and \quad \gamma < - 2/q -2); \\
& q=\infty: &\quad either \quad (\delta<0), \quad or \quad (\delta=0
\quad and \quad \gamma \leqslant -2).
\end{eqnarray*}
\end{enumerate}
\end{corollary}
We can see that for the same values of the hyperparameters
$\alpha$, $\beta$, $\gamma$ functions with the normal model of
wavelet coefficients belong to a wider class of Besov spaces
compared to the functions with Laplacian model. Therefore, if we
want to span a larger set of functions (for fixed values of the
hyperparameters) we need to choose a lighter tail.

\subsection{Distribution without point mass at zero}

Our results apply to the  probability model \eqref{eq:prior} with $\pi_j=1$, i.e. to the case there is no point mass term. Theorem~\ref{th:RegOWT} implies that given distribution $H$,  the necessary and sufficient condition for $f\in B_{pq}^s$ is as follows.
\begin{enumerate}
\item $1\leq p<\infty$, $q<\infty$, $\E|\xi|^p <\infty$: \,\, $\sum_{j=j_0}^{\infty}  2^{j q(s  +1/2)} \tau_j^q < \infty$.
\item $1\leq p<\infty$, $q=\infty$: \quad\,\, $\sup_{j \geqslant j_0}  2^{j (s  +1/2)} \tau_j  < \infty$.
\item $p=\infty$, $1 \leqslant q< Q$: \quad\,\,   $\sum_{j=j_0}^{\infty} 2^{j q(s  +1/2) } \tau_j^q b_j^{q} < \infty$.
\item $p=\infty$, $q=\infty$: \quad\,\, $ \sup_{j \geqslant j_0} [2^{j (s  +1/2)} \tau_j b_j] < \infty$.
\end{enumerate}
That is, it is always necessary to have $\tau_j \ll 2^{-j(s  +1/2)}$, for instance, $\tau_j = j^{-2} 2^{-j(s  +1/2)}$ or $\tau_j =  2^{-j(s  +1/2 + \epsilon)}$ for any positive $\epsilon$. These choices correspond to the case of a non-adaptive prior in wavelet regression, i.e. where the prior distribution depends on the smoothness of the unknown function. An example of the adaptive prior is with $\tau_j = 2^{-j/2}$ with a double exponential or a Gaussian distribution $H$, which implies that the corresponding random function belongs to $B_{p,q}^s$ for all $1\leq p, q\leq\infty$ and $s>0$ with probability 1.



\subsection{A ``trivial'' parametrisation}

 Function $f$ belongs to Besov spaces $B_{pq}^s$ with $s\in (0,r)$ spaces with probability 1 (all Besov spaces that it is possible to recover using wavelets of regularity $r$), if the corresponding wavelet series has a finite number of non-zero wavelet coefficients with probability 1,  for instance, if  the distribution of the wavelet coefficients \eqref{eq:prior} has $\pi_j = 2^{-2j}$ or $\pi_j = j^{-2} 2^{-j}$ and arbitrary $\tau_j$ (case 5 in Theorem~\ref{th:RegOWT}).

\section{Application to non-parametric regression}\label{sec:regression}

Consider a nonparametric regression problem \eqref{eq:regression} with independent Gaussian observation errors $\epsilon_i \sim N(0,\sigma^2)$ where the aim is to estimate the unknown function $f$. A function $f\in L^2[0,1]$ can be represented in a wavelet basis \eqref{eq:WaveletFn}, as discussed in Section~\ref{sec:wavelet}.

The corresponding observed wavelet coefficients satisfy $y_{jk} \mid \theta_{jk}\sim N(\theta_{jk}, \sigma^2 )$ independently, $j=0,\ldots, J-1$, $k=0,\ldots, 2^j-1$ where $\theta_{jk}=w_{jk}\sqrt{n}$ and $J= \lfloor \log_2 n\rfloor$. In the Bayesian approach, the prior distribution \eqref{eq:prior} is put on $\theta_{jk}$ for $j<J$ which corresponds to the prior on $w_{jk}$ with the same $\pi_j$ and scale $\tau_j /\sqrt{n}$ for $j<J$, and $\pi_j = 0$ for $j\geq J$.
In this case, the number of non-zero wavelet coefficients is at most $n$, and hence is bounded for any fixed $n$. An interesting question is a priori membership of Besov spaces by the function with the wavelet coefficients that follow the chosen prior distribution for large sample size, i.e. when $n\to \infty$.
Then, Theorem~\ref{th:RegOWT} can be reformulated as follows.

\begin{theorem}\label{th:RegOWT2} Consider a function $f$ with wavelet coefficients determined by \eqref{eq:WaveletFn} following
distribution (\ref{eq:prior}). Suppose that parameters of the
Besov spaces are restricted to $1 \leqslant p,q \leqslant \infty$,
$0 < s < r$, where $r$ is the regularity of the wavelet and
scaling functions. Denote $s' = s-1/p+1/2$ and $\xi \sim H$.

\begin{enumerate}
\item\label{it1:owt2} $p < \infty$: assume   $\E |\xi|^{p} < \infty$ and  $\pi_j 2^j\rightarrow\infty$ as $j\rightarrow\infty$.

Then, for any fixed finite  $||u_{j_0}||_p$, $\PP\{ f \in B_{p q}^s \} =1$ if and only if
\begin{eqnarray*}
\lim_{n\to \infty} n^{-q/2} \sum_{j=j_0}^{\lfloor \log_2 n\rfloor-1} \left[2^{j (s' +1/p)} \tau_j
\pi_j^{1/p}\right]^q &<& \infty, \quad  1 \leqslant q<\infty;\\
\lim_{n\to \infty} n^{-1/2}\sup_{ j_0 \leq j < \lfloor \log_2 n\rfloor} [2^{j (s' +1/p)} \tau_j \pi_j^{1/p}] &<&
\infty, \quad q=\infty.
\end{eqnarray*}

\item\label{it2:owt2} $p = \infty$. Assume that  $\pi_j 2^j > 1$
 is a sequence monotonically increasing to $\infty$ for large $j$. 
Denote $H_{+}(x) = H(x)-H(-x)$, $x\geqslant 0$, and assume that
either $H_{+}(x)  \in D(e^{-e^{-x}})$ or $H_+(x) \in D(e^{-x^{-\ell}})$,
for some $\ell>0$.

Assume that
\begin{enumerate}
\item   $ [ 2^{j} \pi_j b_j H_+'(b_j)]^{-1} \log j \to 0$ as $j
\to \infty$ if $H_{+}(x) \in D(e^{-e^{-x}})$;
\item  $q<\ell$ \,\,  if $H_+(x) \in D(e^{-x^{-\ell}})$.
\end{enumerate}

Then, for any fixed finite $||u_{j_0}||_p$, $\PP\{ f \in B_{p q}^s \} =1$ if and only if
\begin{eqnarray*}
\lim_{n\to \infty} n^{-q/2} \sum_{j=j_0}^{\lfloor \log_2 n\rfloor-1} [2^{j s' } \tau_j b_j]^{q} &<& \infty, \quad
1 \leqslant q< Q;\\
\quad \lim_{n\to \infty} n^{-1/2}\sup_{j_0 \leq j < \lfloor \log_2 n\rfloor} [2^{j s'} \tau_j
b_j] &<& \infty, \quad q=\infty,
\end{eqnarray*}
where  $Q = \infty$ if $H_{+}(x)  \in D(e^{-e^{-x}})$, and $Q = \ell$
if $H_+(x) \in D(e^{-x^\ell})$.

\item\label{it3:owt2} $1 \leqslant p \leqslant \infty$, $1\leq q <
\infty$. Assume that  $\pi_j 2^j \to const$ as $j\to \infty$ and $\E |\xi|^q < \infty$.
Then, for any fixed finite  $||u_{j_0}||_p$, 
\begin{eqnarray*}
\PP\{ f \in B_{pq}^s \} =1 \quad \Leftrightarrow \quad
\lim_{n\to \infty} n^{-q/2}\sum_{j=j_0}^{\lfloor \log_2 n\rfloor-1} [2^{j s'} \tau_j]^q < \infty.
\end{eqnarray*}

\item\label{it4:owt2W} $1 \leqslant p \leqslant \infty$, $q =
\infty$, $\pi_j 2^j \to const$ as $j\to \infty$.  Assume that $M(j) = \tau_j^{-1}
2^{-j s'}$ is a monotonically increasing function for  $j\geqslant
\tilde{j}$ for some $\tilde{j} \geqslant j_0$, and that $M(x)$ is extended to $[\tilde{j}, \infty)$ in such a way that function $M(x)$ (and hence $M^{-1}(x)$ is monotonically increasing.

Then, for any fixed finite $||u_{j_0}||_p$, 
$$\PP\{ f \in B_{pq}^s \} =1 \,\, \Leftrightarrow \,\, \exists c > 0 :\,\, \E
M^{-1}(|\xi|/c) < \infty.$$

\item\label{it5:owt2}  $\sum_{j=j_0}^{\infty} \pi_j 2^j < \infty$. Assume that $\PP(|\xi|<\infty)=1$. 
 Then, for any fixed finite $||u_{j_0}||_p$, $\PP\{ f \in
B_{pq}^s \} =1$ for $p,q \geq 1$ and $s\in (0,r)$.

\end{enumerate}

\end{theorem}

The proof is obtained by replacing the infinite sums  in the proof of Theorem~\ref{th:RegOWT} by the sums up to $\lfloor \log_2 n\rfloor-1$ and using  the scale $\tau_j /\sqrt{n}$ instead of $\tau_j$, and taking the limit as $n\to \infty$.

The last condition implies that if $\pi_j$ is small enough so that $  \sum_{j=j_0}^{\infty} \pi_j 2^j < \infty$, then $f \in
B_{pq}^s$ almost surely for all values of parameters $p,q \geq 1$ and $s\in (0,r)$. Now we give examples of non-trivial probability models that correspond to the first four cases.

\citet{p:joh-sil-05} proved global minimax optimality of Bayesian thresholding estimators with prior distributions $H$ having tails not heavier than Cauchy and not lighter than exponential, with $\tau_j \equiv \tau$ for all $j$. We start with considering these  prior distributions.
\begin{example} Laplacian prior: $h(x) = \frac 1 2 e^{-|x|}$, $x\in \mathbb{R}$.

  \begin{enumerate}
\item  $p < \infty$  and  $\pi_j 2^j\rightarrow\infty$ as $j\rightarrow\infty$.
The necessary condition  $n^{ (s+1/2)} \pi_J^{1/p} = O(n^{1/2})$ for the sum and the supremum to be finite, which can be rewritten as  $\pi_J = O(n^{-sp})$, implies that for large enough $j$, we can take $\pi_j = O(2^{-jsp})$ to ensure that $f\in B_{pq}^s$ almost surely.

Consider the case $c   2^{-\beta j} \leq \pi_j \leq C  2^{-\beta j}$ for some $\beta \in (0,1)$, $c, C>0$ for large enough $j$.  For $1 \leqslant q<\infty$,  the necessary and sufficient condition is
\begin{eqnarray*}
&&\lim_{n\to \infty} n^{-q/2} \sum_{j=j_0}^{\lfloor \log_2 n\rfloor-1}  2^{j q(s  +1/2 -\beta/p)} \\
&\asymp&  \lim_{n\to \infty}  \sign(s  +1/2 -\beta/p) [n^{  q(s   -\beta/p)} - n^{-q/2}2^{j_0 q(s  +1/2 -\beta/p)}] <\infty,
\end{eqnarray*}
which holds if and only if $ s  \leq \beta/p$. Note that this condition differs from the one obtained for infinite models given in Section~\ref{sec:mainBesov}  which is $s  < \beta/p -1/2$ for the considered case $\alpha =0$.


For $q=\infty$, the necessary and sufficient condition is
\begin{eqnarray*}
\lim_{n\to \infty}   [n^{  (s   -\beta/p)}I(s  +1/2 >\beta/p) + n^{-1/2}2^{j_0 (s  +1/2 -\beta/p)} I(s  +1/2 \leq \beta/p )] <\infty
\end{eqnarray*}
which holds if and only if $ s  \leq \beta/p-1/2$, and which is the same as in the infinite dimensional case.

\item   $p = \infty$.  Since $H_{+}(x) = 1 - e^{-x} \in D(e^{-e^{-x}})$, $ b_j  =  \log (\pi_j 2^j)$ is a monotonically increasing sequence, $ \log(\pi_j 2^j)/\log j\to \infty$ for $j\to \infty$. Then, necessary and sufficient condition is not satisfied.

\item   $1 \leqslant p \leqslant \infty$, $q <\infty$, $\pi_j 2^j \to const$. Then,
the necessary and sufficient condition holds if $s\leq 1/p$. For the infinite series considered in Section~\ref{sec:mainBesov}, the corresponding condition  $s < 1/p-1/2$ is stronger.

\item   $1 \leqslant p \leqslant \infty$, $q =\infty$, $\pi_j 2^j \to const$.  Here $M(j) =  2^{-j s'}$ is a monotonically increasing function for  $j\geqslant  j_0$ if and only if $s' = s+1/2-1/p <0$, i.e. $s  < 1/p-1/2$. 

\item  If $  \sum_{j=j_0}^{\infty} \pi_j 2^j < \infty$ and $\PP(|\xi|<\infty)=1$, then, for any fixed finite $||u_{j_0}||_p$, $\PP\{ f \in
B_{pq}^s \} =1$ for $p,q \geq 1$ and $s\in (0,r)$.

To summarise: for the Laplace prior distribution, we must have either $\pi_j$ such that $\pi_j 2^j \to \infty$ and $\pi_j = O(2^{-jsp})$ as $j\to \infty$, or $\sum_{j=j_0}^{\infty} \pi_j 2^j < \infty$.

If $\pi_j =C 2^{-\beta j}(1+o(1))$ as $j\to \infty$ for some $\beta \in (0,1)$ then the function belongs to the following Besov spaces $\{B_{p,q}^s: \, p<\infty\,  \& \, s  \leq \beta/p\}$. This includes the set of generalised functions with $s<1/p$.

\end{enumerate}

Therefore, for the Laplacian distribution $H$, given parameters $(\pi_j)$ and $\tau_j = \tau_0$, the prior set is wider in the case $1 \leqslant p \leqslant \infty$, $q <\infty$, $\pi_j 2^j \to const$.

If we also assume some  scale $\tau_j$ decreasing to $0$ for large $j$, then the set of covered Besov spaces is larger.

\end{example}

\begin{example} Consider the case $H$ is a Cauchy distribution with the denssity $h(x) = \frac 1 {\pi(1+x^2)}$, $x\in \mathbb{R}$.
\begin{enumerate}
\item Case \ref{it1:owt2}: $p < \infty$    and  $\pi_j 2^j\rightarrow\infty$ as $j\rightarrow\infty$:
assumption $\E |\xi|^{p} < \infty$ is not satisfied for any $p \geq 1$.

\item Case \ref{it2:owt2}: $p = \infty$. Assume that  $\pi_j 2^j > 1$
 is a monotonically increasing function of $j$ for large $j$.
In this case, $H_+(x) \in D(e^{-x^\ell})$ with $\ell = 2$, and
$$
b_j = \tan\left( \frac{\pi} 2 (1 - (\pi_j
2^j)^{-1})\right) \approx \frac{\pi} 2 (\pi_j 2^j)^{-1}
$$

Then, the necessary and sufficient condition implies that we must have  $n^{s} (\pi_J n)^{-1} =o(1)$ which can hold only if $s<1$ and $\pi_j \gg 2^{-j(1-s)}$.

For $\pi_j =C 2^{-\beta j}(1+o(1))$ with $\beta \in (0,1)$, we must have
\begin{eqnarray*}
\lim_{n\to \infty} 
[n^{q (s-1+\beta)} - n^{-q/2}  2^{j_0 q (s-1/2+\beta)}] &<& \infty, \quad
1 \leqslant q< 2;\\
\quad \lim_{n\to \infty} n^{-1/2}\sup_{j_0 \leq j < \lfloor \log_2 n\rfloor} [2^{j (s-1/2 +\beta) }] &<& \infty, \quad q=\infty,
\end{eqnarray*}
  which are finite if and only if $s\leq 1-\beta$.

For $\pi_j =C j^{-a}(1+o(1))$ with $a>0$, we must have
\begin{eqnarray*}
\lim_{n\to \infty} 
[n^{q (s-1)}[\log n]^{aq+1} - n^{-q/2}  2^{j_0 q (s-1/2)} j_0^{aq+1}] &<& \infty, \quad
1 \leqslant q< 2;\\
\quad \lim_{n\to \infty} n^{-1/2}\sup_{j_0 \leq j < \lfloor \log_2 n\rfloor} [2^{j (s-1/2)} j^{a}] &<& \infty, \quad q=\infty,
\end{eqnarray*}
  which are finite if and only if $s<1$.


In cases \ref{it3:owt2} and \ref{it4:owt2W} ($\pi_j 2^j \to const$ as $j\to \infty$), the necessary and sufficient conditions are not satisfied if $\tau_j = const$.

\item Case \ref{it5:owt2}: $  \sum_{j=j_0}^{\infty} \pi_j 2^j < \infty$, $\PP(|\xi|<\infty)=1$.
 Then, for any fixed finite $||u_{j_0}||_p$, $\PP\{ f \in B_{pq}^s \} =1$ for $p,q \geq 1$ and $s\in (0,r)$.


\end{enumerate}

Therefore, in the non-trivial setting, the Cauchy prior distribution with $\pi_j 2^j$ monotonically increasing with $j$, ensures that the function belongs to $B_{\infty, q}^{s}$ with $s<1$.

\end{example}

\section{Continuous wavelet transform}\label{sec:CWT}

In this section we extend the result stated in Section \ref{sec:mainBesov} to the continuous wavelet transform under the same type of parametrisation.

\subsection{Continuous wavelet transform}\label{sec:ContWavDef}

A wavelet basis $\psi_{jk}(x) = 2^{j/2}\psi (2^j x - k)$ with
dyadic shift and
 scale can be extended to a set of functions with arbitrary shift and scale, i.e.
$\psi_{a, b} (x) = |a|^{1/2} \psi(a(x-b))$, $a\in\mathbb{R}\setminus \{0\}$, $b\in\mathbb{R}$.
This set of functions can be used to perform the continuous wavelet transform
$T^{wav}: \quad L^2(\mathbb{R})\tends L^2(\mathbb{R}\setminus \{0\} \times
\mathbb{R})$ in the following way:
\begin{equation}
T^{wav}f (a,b) = \langle f, \psi_{a,b} \rangle_{L^2},
\end{equation}
where $\langle\cdot, \cdot \rangle_{L^2}$ is the scalar product in
Hilbert space $L^2(\mathbb{R})$. Since $a$ can be interpreted as a
frequency we can restrict ourselves to the case $a>0$. This
corresponds to the case where both $f$ and $\psi$ are
``analytical'' signals, i.e. if $\supp(\hat{f})\subset (0,
\infty)$, $\supp(\hat{\psi}) \subset (0, \infty)$, where $\hat{f}$
is the Fourier transform of $f$, which implies $T^{wav}f(a, b) =0$
for $a<0$.

 If the {\it admissibility condition}: $C_{\psi} = 2\pi
\int_0^{\infty} \frac{|\hat{\psi}(\xi)|^2}{|\xi|} d\xi 
< \infty$
 is satisfied it is possible to recover a function from its
continuous wavelet transform:
\begin{equation}\label{eq:ResId1}
f(x) = C_{\psi}^{-1} \int_{0}^{\infty}da \int_{\mathbb{R}}db
T^{wav}f (a,b) \psi_{a,b}(x).
\end{equation}

In the next section we consider a probabilistic model for functions based on continuous wavelet transform.

\subsection {Probabilistic model}\label{sec:modelCWT}

We assume that wavelet and scaling functions $\psi$ and $\phi$
have compact support $[0, 1]$ under periodic boundary condition and are of regularity $r$.
We model a function $f$ on $[0, 1]$ as a sum of high and low
frequency components $f_w$ and $f_0$:
\begin{eqnarray}\label{eq:ContModel}
f(x) = f_0(x) + f_w(x) = \sum_{i=1}^{M} \eta_{\lambda_i}
\phi_{\lambda_i}(x) + \sum_{\lambda\in S} \omega_{\lambda}
\psi_{\lambda},
\end{eqnarray}
where $\psi_{\lambda}(x) = a^{1/2} \psi(a(x-b))$, $\lambda=(a,b)$,
$a\geqslant a_0=2^{j_0}$, $b\in[0,1]$, and similarly,
$\phi_{\lambda}(x) = a^{1/2} \phi(a(x-b))$, $M<\infty$ and
$\lambda_i$ are such that $a_i\leqslant a_0=2^{j_0}$, $b_i\in[0,
1]$ where $a_0$ is at least twice the length of the support of
functions $\psi$ and $\phi$. The coarse component of the function,
$f_0$,  is considered to be a finite linear combination of scaling
functions with real-valued coefficients $\eta_{\lambda_i}$. For
the high frequency component, we adapt the probabilistic model
(\ref{eq:prior})
  used for the wavelet coefficients of the orthogonal wavelet transform
to the continuous wavelet coefficients
$\omega_{\lambda}$ and the set of their indices $S$.

Defining set $S$ of wavelet indices $\lambda = (a, b) \in [a_0,
\infty)\times[0,1]$ corresponds to selecting a set of indices
where the wavelet coefficients are non-zero. Here it is modelled
as a Poisson process with intensity $\mu_{\lambda}$. The intensity
of the Poisson process $\mu_{\lambda}$ determines the number of
elements of the process around location $\lambda$. This is an
analogue of parameter $\pi_j$ we used for the orthogonal model
(\ref{eq:prior}) which is the proportion of non-zero wavelet
coefficients at resolution level $j$.

Since the distribution of the
wavelet coefficients  $\omega_{\lambda}$ depend on the Poisson process $S$,
 we model the distribution of wavelet coefficients $\omega_{\lambda} $ conditioned
on the Poisson process $S$ in the similar way as the distribution
of non-zero orthogonal wavelet coefficients. Therefore we assume
that $\omega_{\lambda}$ are conditionally independent given $S$
and have distribution $H_{\lambda}$:
\begin{equation}
\omega_{\lambda} \vert S \sim   H_{\lambda} (x),
\end{equation}
where $H_{\lambda} (x) = H(x/\tau_{\lambda})$ is a distribution
function, continuous at $x=0$.

In the next sections we study regularity of the random functions following stochastic expansion~\eqref{eq:ContModel} that can be easily obtained from Theorem~\ref{th:RegOWT}. A more comprehensive study using different tools will be done separately (see Section~\ref{sec:Discussion} for a discussion).

\subsection{Regularity properties under simple parametrisation}

In this section we consider the following parametrisation of variance and intensity of the
Poisson process: $\mu_{\lambda} = C_{\mu} a^{-\beta}$,
$\tau_{\lambda}^2 = C_{\tau} a^{-\alpha}$, $\beta\geqslant 0$,
 $\alpha\geqslant 0$, $C_{\mu}>0$, $C_{\tau}>0$. Hyperparameters
$\alpha$ and $\beta$ have the same interpretation as their
counterparts in the model for the orthogonal wavelet coefficients under the simple parametrisation
in Section~\ref{sec:mainBesov}.

Now we connect the hyperparameters of the model for
wavelet coefficients $\alpha$ and $\beta$ with parameters of the
Besov spaces $p$ and $s$. We split the necessary assumptions into three groups.

\par
{\bf Assumption W}: {\it
 $\psi$ is a compact-supported wavelet function of regularity $r$,
$\psi^{(r)}\in \mathbb{C}^{\rho}$, where $\rho\in (0,1)$ is the
exponent of H\"older continuity; $0 < s < r$ and $r+\rho>1/p-1/2$.
}

We need the regularity of the wavelet function to be greater than
the parameter $s$ of the Besov space in order to be able to use
the equivalence of the Besov sequence norm $b_{p,q}^s$ on wavelet
coefficients and the Besov norm of the function with these
coefficients. We can see that this assumption includes the assumption H
 used in Theorem \ref{th:mainBesov}, with addition  of an extra
regularity condition on the wavelet function which is due to
complexity of the wavelet indices.

\par

{\bf Assumption B}: {\it Assume that $f$ follows \eqref{eq:ContModel} with
\begin{equation}
\omega_{\lambda} \vert S \sim   H_{\lambda} (x),
\end{equation}
where $S$ is a Poisson process on $\lambda=(a, b) \in[a_0,
\infty)\times[0,1]$, $a_0=2^{j_0} \geqslant 2 \max\{ |\supp(\psi)|,
|\supp(\phi)| \}$, with intensity $\mu(\lambda) = C_{\mu}
a^{-\beta}$, $0\leqslant \beta\leqslant 1$, and $H_{\lambda} (x) =
H(x/\tau_{\lambda})$ has finite variance $\tau_{\lambda}^2 =
C_{\tau} a^{-\alpha }$, $\alpha\geqslant 0$. We also assume that
$\beta+\alpha>0$. }

The latter assumption means that at least one of the parameters
$\mu_{\lambda}$ or $\tau^2_{\lambda}$ of the model depends on
$\lambda$ and thus we exclude the case $\beta=\alpha=0$ which
corresponds to functions which almost surely do not belong to
Besov spaces, as for the orthogonal wavelet transform.

The assumption about distribution $H$, {\bf Assumption H}, is the
same as in case of the orthogonal wavelet transform given in
Section~\ref{sec:mainBesov}.


Similarly to the orthogonal
model, in the case $p = \infty$ and $\beta < 1$ we consider two
particular cases: a distribution with the power exponential tail
in the theorem because the criterion for this type of
distributions coincides with the criterion for other cases, and a
distribution with polynomial tail in the proposition after the
theorem.
Now we can give the criterion linking the model for wavelet
coefficients and the parameters of the Besov spaces.

\begin{theorem}\label{th:ContBesov}
Assume that continuous wavelet transform (\ref{eq:ContModel}) of function
$f: [0,1] \rightarrow \mathbb{R}$ follows assumptions W, B, H, for $1\leqslant p, q \leqslant
\infty$ and $r+\rho>(1+\alpha)/2$.  Under Assumption H 2(a) (polynomial tail), we also assume that $\ell >2/(r+\rho+1/2)$.

Then, $f\in B^s_{p,q}$ almost surely if and only if
\begin{eqnarray*}
s  &\leqslant& (\alpha-1)/2+ \beta/p,  \quad\quad \text{if }\, p<\infty, q=\infty, \beta<1,\\
s  &<&   (\alpha-1)/2+ \beta/p+\delta_H,\quad \text{otherwise},
\end{eqnarray*}
where $\delta_H$ is defined by \eqref{eq:defDeltaH}.
\end{theorem}
This theorem follows from Theorem~\ref{th:ContBesov_Gen} with $\mu(a) = C_\mu a^{-\beta}$ and $\tau(a) = \sqrt{C_\tau} a^{-\alpha/2}$ except for the case $p=\infty$ and $\beta <1$ which is proved in Appendix~\ref{sec:ProofCWT_prop}. If $H$ is normal, the necessary and sufficient condition coincides with the one given in \citet{p:abr-sap-silv-00}.

 We can see that the necessary and
sufficient criteria for a function to belong to the Besov spaces
are the same for the
continuous as for the orthogonal wavelet
transform. Note that similarly to the case of the orthogonal
wavelet transform, in case $\beta > 1$ the expected number
$\int_{a_0}^{\infty} \int_0^1 a^{-\beta} da db$ of  non-zero
wavelet coefficients $\omega_{\lambda}$ is finite. Therefore, for
$\beta > 1$, function $f$ belongs to the same Besov space as the
wavelet function, i.e. $B_{p,q}^s$ with $1\leqslant p,q \leqslant
\infty$ and $0<s<r$.


\subsection{Regularity properties under general parametrisation}

Now we consider a more general parametrisation where we
 assume that the intensity $\mu_\lambda$ and the scale $\tau_\lambda$ are independent of the shifting parameter
$b$ and decrease as the scaling parameter $a$ increases. Assumption $B$ is replaced with Assumption $B'$.

{\bf Assumption $B'$}: {\it Assume that $f$ follows \eqref{eq:ContModel} with
\begin{equation}
\omega_{\lambda} \vert S \sim   H_{\lambda} (x),
\end{equation}
where $S$ is a Poisson process on $\lambda=(a, b) \in[a_0,
\infty)\times[0,1]$, $a_0=2^{j_0} \geqslant 2 \max\{ |\supp(\psi)|,
|\supp(\phi)| \}$, with intensity $\mu(\lambda) = \mu(a)$, and $H_{\lambda} (x) =
H(x/\tau_{\lambda})$ has finite variance $\tau_{\lambda}^2 =
[\tau(a)]^2$ such that
\begin{enumerate}
\item[1)]  functions $\mu(a)\geq 0$, $\tau(a) \geq 0$ are decreasing,

\item[2)] $\mu(a) \to 0$ or $\tau(a) \to 0$ as $a\to \infty$,

\item[3)] function $\tau(a)$ is continuous,

\item[4)] for $p<\infty$,  $\int_0^{a_0} a^{p(\rho+r + 1/2)-1} [\tau(a)]^p \mu(a) da<\infty$.

\end{enumerate}
 }

Note that in case $\tau(a) = C a^{-\alpha/2}$ considered by \citet{p:abr-sap-silv-00}, their assumption $r+\rho > (1+\alpha)/2 $ is stated here as assumption $B'$4).

\begin{theorem}\label{th:ContBesov_Gen}
Assume that continuous wavelet transform (\ref{eq:ContModel}) of function
$f: [0,1] \rightarrow \mathbb{R}$ follows
assumptions W and $B'$  for $1\leqslant p < \infty$, $1\leq q \leq \infty$. Denote $s' = s-1/p+1/2$ and $\xi \sim H$.

\begin{enumerate}
\item\label{it1:cwt} $1\leq q < \infty$, $1\leq p<\infty$, assume that $\E |\xi|^{p} < \infty$ and

either    $\mu(2^j) 2^j$ increases to $\infty$ as $j\rightarrow\infty$ monotonically   for large  $j$,

or    $\mu(2^j)  2^j \to const$ as $j\to \infty$ and $\E |\xi|^{q} < \infty$.

Then, for any fixed finite $||u_{j_0}||_p$,
\begin{eqnarray*}
\PP\{ f \in B_{p q}^s \} =1 \quad \Leftrightarrow \quad
\sum_{j=j_0}^{\infty} \left[2^{j (s' +1/p)} \tau_j
[\mu(2^j)]^{1/p}\right]^q < \infty.
\end{eqnarray*}

\item\label{it2:cwt} $q=\infty$, $1 \leqslant p < \infty$:
assume  $\E |\xi|^{p} < \infty$, $\mu(2^j) 2^j\rightarrow\infty$ as $j\rightarrow\infty$ and $\mu(2^j) 2^j$ is a monotonically increasing sequence  for large enough $j$.
 Then, for any fixed finite $||u_{j_0}||_p$,
$$\PP\{ f \in B_{pq}^s \} =1\quad \Leftrightarrow\quad \sup_{j \geqslant j_0} [2^{j (s' +1/p)} \tau_j [\mu(2^j)]^{1/p}] < \infty.$$

\end{enumerate}
\end{theorem}
Proof of the theorem is given in Appendix \ref{sec:ProofCWT}. Note that we consider only the case $p<\infty$, since in this case the maximum $\max_{k}|w_{jk}|$ is no longer a maximum of independent random variables which affects the necessary and sufficient condition and hence  requires a different type of proof. This case is being investigated as a part of ongoing work on the necessary and sufficient condition on $(\tau, \mu, H)$ for $f\in B_{pq}^s$ almost surely.

\section{Discussion}\label{sec:Discussion}

The main results of this paper are necessary and sufficient conditions establishing the connection between
the regularity properties of a function in terms of its Besov norm and a probabilistic model for its wavelet coefficients.
The necessary and sufficient conditions for the two parameter models are the same for the orthogonal
and the continuous wavelet transform, however for the continuous transform we make an additional assumption of
finite variance of the distribution of the continuous wavelet coefficients and also we need a more detailed
information about the regularity of the wavelet function.



The obtained results can  be used to compare the effect of
different tail behaviour of the wavelet coefficients on the
regularity properties of the function. For instance, it follows
that with probability 1 the Besov spaces contain more functions
whose wavelet coefficients have distribution with normal or
exponential tail than those with polynomial tail. It also follows
from the assumptions of the theorems that the more moments a
distribution has, the wider is the set of the corresponding Besov
spaces. In case $p=\infty$ and  with  the number of nonzero
wavelet coefficients increasing at higher levels (e.g. $\beta <1$
if $\pi_j\sim c2^{-\beta j}$), the necessary and sufficient
condition depends on the tail of the distribution directly, i.e.
the heavier the tail is, the narrower is the set of the
corresponding Besov spaces as we saw on the example of $t$ versus
normal and Laplacian distributions.

An important application of the results of this paper is to study a priori Besov membership of Bayesian wavelet estimators.
We study the actual prior model that is used in Bayesian nonparametric wavelet regression, and show that the corresponding a priori class of function is wider than it is under the infinite dimensional ``idealisation'' of this prior that has been previously investigated in the literature.
Also, these results can be used in Bayesian regression modelling
 to specify the probabilistic distribution of wavelet coefficients if there is some
 information available about the regularity of the function of interest.
If the regularity of function is known a priori, it can be used to
specify the hyperparameters of the prior model for wavelet
coefficients and to choose the distribution of non-zero wavelet
coefficients. The results obtained in this paper can be applied to
many prior models, and thus a necessary and sufficient condition
of a priori Besov membership of a function can be compared to
other properties of the Bayesian models, e.g. frequentist
optimality over the Besov spaces, thus making a more precise
statement compared to that in \cite{p:pensky-2003} and identifying the Besov spaces that are not covered almost surely by the chosen prior distribution.

The results for the continuous wavelet transform can be applied to
a wavelet transform on an irregular grid which corresponds to
various applications where irregularity of the grid can be
modelled as a Poisson process. This result can be extended
to other probability models of irregularity.

Another question of interest, perhaps more from a probabilistic perspective, is a necessary and sufficient condition for the unknown function to belong to Besov spaces on the hyperparameters and distribution $H$  simultaneously, rather than for a given $H$. This is beyond the scope of this paper and is work in progress.


\appendix

\section{Law of large numbers for a sequence of random length}

\begin{lemma}\label{lem:LLN}
 Let $\xi_{jk}$ be independent  positive random variables $k=0,1,\dots, 2^j-1$, $j=j_0, j_0+1, \ldots$  that have distribution function  $G(x)$ and finite mean: $\mu =\int_0^{\infty} x d G(x) < \infty$.

Let $N_j \sim Bin(2^j, \pi_j)$ be independent random variables, independent of $(\xi_{jk})$, such that $\E N_j = n_j \stackrel{def}{=} 2^j \pi_j \to \infty$ as $j\to \infty$ such that sequence $n_j$ monotonically increases  for large enough $j$.

Denote $S_j =  \sum_{k=1}^{N_j} \xi_{jk}$. Then,
$\frac{S_j}{n_j} \rightarrow
\mu$ with probability 1 as $j \rightarrow \infty$.

\end{lemma}
This is the law of large numbers for sequences with a random number of coefficients.

\begin{proof}[Proof of Lemma \ref{lem:LLN}]
We generalise the proof of the law of large number of sequences with a fixed number of elements given in \citet{p:Etemadi}. 

1. Denote $Y_{jk} = \xi_{jk} I(\xi_{jk} \leq 2^j)$. Since, for all $k$,
$$
\sum_{j=j_0}^{\infty} \sum_{k=0}^{2^j-1} \PP(\xi_{jk} > 2^j )=\sum_{j=j_0}^{\infty}  2^j \PP(\xi_{jk} > 2^j )\leq \int_{2^{j_0}}^{\infty}  \PP(\xi_{jk} > t) dt \leq  \E \xi_{jk} <\infty,
$$
we have $\PP(Y_{jk} \neq \xi_{jk} \, i.o. \, j,k) =0$. Therefore $ |S_j(\omega)  - T_j(\omega) |\leq R(\omega)<\infty$ a.s. for all $j$ where $T_j(\omega)=  \sum_{k=1}^{N_j} Y_{jk}(\omega)$, and therefore it is sufficient to show that $T_j/n_j \to \mu$ a.s.

\vspace{0.2cm}

2. Now we study the mean and variance of $T_j$:
\begin{eqnarray*}
\E(T_j) &=&  \E \left(\E\left(\sum_{k=1}^{N_j} Y_{jk} \mid N_j\right)\right) =  \E N_j \E  Y_{jk} = n_j \E(Y_{jk} ),
\end{eqnarray*}
and
\begin{eqnarray*}
\Var(T_j) &=&  \E \Var(\sum_{k=1}^{N_j} Y_{jk} \mid N_j) +  \Var[ \E (\sum_{k=1}^{N_j} Y_{jk} \mid N_j)]\\
 &=&  \E N_j \Var(  Y_{jk} ) +  \Var[N_j \E Y_{jk}] = n_j \Var(Y_{jk} ) +  n_j (1-\pi_j) [\E Y_{jk}]^2 \leq n_j \E Y_{jk}^2.
\end{eqnarray*}

\vspace{0.2cm}

3. Take $m(j)$ such that $n_{m(j)} = [\alpha^j]$ for some $\alpha >1$. The Chebyshev inequality implies that for any $\epsilon > 0$,
\begin{eqnarray*}
  \sum_{j=j_0}^{\infty}  \PP( |T_{m(j)} - \E T_{m(j)}| > \epsilon n_{m(j)} ) &\leq&  \epsilon^{-2} \sum_{j=j_0}^{\infty}  \Var(T_{m(j)})  n_{m(j)}^{-2}\\
&\leq& \epsilon^{-2} \sum_{j=j_0}^{\infty}   n_{m(j)}^{-1} \E (Y_{m(j), k}^2).
\end{eqnarray*}
Now,
\begin{eqnarray*}
 \sum_{j=j_0}^{\infty}  n_{m(j)}^{-1} \E (Y_{m(j), k}^2) &=& \sum_{j=j_0}^{\infty} n_{m(j)}^{-1} \int_0^{\infty} 2 y \PP(Y_{m(j), k} >y) dy  \\
 &\leq & \int_0^{\infty} 2 y  \left[ \sum_{j=j_0}^{\infty}   n_{m(j)}^{-1}  I( y < 2^{m(j)}) \right] \PP(\xi_{1, 1} >y) dy   \\
 &\leq & 2 [1-1/\alpha]^{-1}\E\xi_{1, 1}  < \infty
\end{eqnarray*}
since
\begin{eqnarray*}
  y  \sum_{j=j_0}^{\infty}   n_{m(j)}^{-1}  I( y < 2^{m(j)}) &\leq&  y  \sum_{j=j_0}^{\infty}   \alpha^{-j}  I( 2^{m(j)} \pi_{m(j)} > y)  \leq y  \sum_{j\geq j_0 \, \& \, j > \log_{\alpha} y} \alpha^{-j}  \\ &\leq&   C y   \, \alpha^{-\max(\log_\alpha y, j_0)} \leq C.
\end{eqnarray*}

Since $\epsilon >0$ is arbitrary, this implies that $n_{m(j)}^{-1}(T_{m(j)} - \E T_{m(j)}) \to 0$ as $j\to \infty$. The dominated convergence theorem implies $\E Y_{jk} \to \E \xi_{jk} $ as $j\to \infty$, hence $n_{m(j)}^{-1}T_{m(j)} \to \E \xi_{1,1} = \mu$ a.s. For the intermediate values $m(j) < \ell < m(j+1) $, we use
$$
\frac{T_{m(j)}}{ m(j+1) } \leq \frac{T_{\ell}}{ \ell } \leq \frac{T_{m(j+1)}}{ m(j) }
$$
due to $Y_{jk}\geq 0$, and since $m(j+1)/m(j) = [\alpha^{j+1}]/[\alpha^{j}] \to \alpha$, we have
$$
\frac 1 {\alpha} \mu \leq \liminf_{\ell \to \infty} \frac{T_{\ell}}{ \ell } \leq \limsup_{\ell \to \infty} \frac{T_{\ell}}{ \ell } \leq \alpha \mu.
$$
Since $\alpha >1$ is arbitrary, the proof is complete.


\end{proof}

\section{Proof of Theorem 2}\label{seq:ProofDWT}


Below we consider vector $z_j$ of
normalised wavelet coefficients $z_{jk} = \tau_{j}^{-1} w_{jk}$,
$k = 0, \dots, 2^j-1$. Its absolute value $|z_{jk}|$ has the
distribution function $1 -\pi_j [1-H(x) + H(-x)]$, $x\geqslant 0$.
We will denote generic positive constants by $C$
which may take different values even within a single equation.



\ref{it1:owt}.
Let $\nu_p$ be the $p$th absolute moment of a
distribution $H$, $\nu_p < \infty$. Then the mean of $||z_j||_p^p$
can be written as
\begin{eqnarray*}
\E||z_j||_p^p &=& \E \left( \sum_{k=0}^{2^j-1} |z_{jk}|^p \right) =
\sum_{k=0}^{2^j-1} \E|z_{jk}|^p =\nu_p  \pi_j 2^j.
\end{eqnarray*}
Under the assumptions of the theorem, by Lemma~\ref{lem:LLN},  $(\pi_j 2^j)^{-1}  ||z_j ||_p^p \rightarrow \nu_p$ almost surely as $j
\rightarrow \infty$. Below we abbreviate ``almost surely'' by a.s.

 For a finite $q$ the Besov norm of wavelet coefficients can be
represented in terms of $||z_j||_p$:
\begin{eqnarray*}
||w||_{b^s_{p,q}} &=& ||u_{j_0}||_p+ \left[\sum_{j=j_0}^{\infty}
2^{js' q} \left( \sum_{k=0}^{2^j-1} |z_{jk}|^p \tau_j^p
\right)^{q/p} \right]^{1/q}
\end{eqnarray*}
which is finite a.s. if and only if the series
$\sum_{j=j_0}^{\infty} 2^{jq (s'+1/p)} \tau_j^q \pi_j^{q/p}$ is
finite.

For infinite $q$ the Besov norm of wavelet coefficients is represented in terms
of $|| z_j||_p$ in the following way:
\begin{eqnarray*}
||w||_{b^s_{p,\infty}} &=& ||u_{j_0}||_p + \sup_{j \geqslant
j_0} \left\{ 2^{js'} \left( \sum_{k=0}^{2^j-1} |w_{jk}|^p
\right)^{1/p} \right\} \\&=& ||u_{j_0}||_p + \sup_{j \geqslant j_0}
\left\{ 2^{j (s'+1/p)} \tau_j \pi_j^{1/p} [(\pi_j 2^j)^{-1} ||
z_j||_p] \right\}.
\end{eqnarray*}
The supremum is finite a.s. if and only if $\lim_{j\to
\infty} 2^{j (s'+1/p)} \tau_j \pi_j^{1/p} < \infty$.

Now we apply Theorem 2 of \citet{p:don-joh-98}
that under the assumptions of the theorem the finiteness of the
Besov norm of the wavelet coefficients is equivalent to the
finiteness of the Besov norm of the function.

\par

\ref{it2:owt}. The Besov sequence norm in  case $p=\infty$ is expressed in terms
of the following random variable: $\xi_j = || z_j||_{\infty} =
\max_{k=0, \dots, 2^j-1} (|z_{jk}|)$. The distribution function of
$|z_{jk}|$ is
\begin{eqnarray*}
\PP\{ |z_{jk}| < x \} = 1-\pi_j + \pi_j H(x) - \pi_j H(-x) = 1 - \pi_j[1-H_+(x)], \, x>0,
\end{eqnarray*}
with $H_{+}(x) = H(x)-H(-x)$,  and the distribution function of
$\xi_j$, due to independence of $w_{jk}$, is
\begin{eqnarray*}
F_{\xi_j}(x) &=& [\PP\{ |z_{jk}| < x \}]^{2^j}=  \left[1 - \pi_j(1 -
H_+(x))\right]^{2^j} \\ &=& \exp\{-\pi_j 2^{j} (1 -
H_+(x))\}\left[1 +O(1)\pi_j^{-2} 2^{-j} (1 - H_+(x))^2\right],
\quad j\rightarrow \infty.
\end{eqnarray*}

To find the asymptotic distribution of $\xi_j$ as $j\to\infty$
it appears that we cannot apply Extreme Value Theory directly
because the distribution of $z_{jk}$ depends on $j$. Nevertheless
the proof of the Theorem 1.6.2   \citep[p.17]{b:led-83} stating the asymptotic distribution of the maximum under
different conditions remains valid in the case $F_{|z_{jk}|}(x) = 1 - \pi_j [1 -
H_+(x)]$ (for sufficiently large $j$) since $\pi_j 2^j$, and thus
$F_{|z_{jk}|}(x)$, depends on $j$ monotonically.

Now we study convergence of the sequence norm of the wavelet
coefficients separately for distributions from the domains of
attraction of the first $D(e^{-e^{-x}})$ and second $D(e^{-x^l})$
types of the extreme value distributions.

a) {\bf $H_{+}(x) = H(x) - H(-x) \in D(e^{-e^{-x}})$, $x\geqslant 0$.} For this type of
distributions, the asymptotic distribution of $||z_{j}||_{\infty}$ is
\begin{eqnarray*}
\PP\{||z_{j}||_{\infty} > a_j x + b_j\} \to 1 - \exp\{ -e^{-x} \},
\end{eqnarray*}
and the constants can be chosen in the following way: 
\begin{eqnarray*}
b_j = H_{+}^{-1}\left(1 - 1/(\pi_j 2^j)\right), \quad a_j =
g(b_j),
\end{eqnarray*}
where $g(x) = (1 - H_+(x))/H_+'(x)$. Note that one of the properties of the distributions from the domain of attraction
of the first type of the extreme value distributions is that $g'(x) \to 0$ as $x \to \infty$.
This implies (by L'Hospital rule) that $g(x)/x \to 0$ as $x\to \infty$, i.e. that $b_j/a_j\to \infty$ as $j\to \infty$.

 If we show that for any $\epsilon>0$,
\begin{equation}\label{eq:asSum}
\sum_{j=j_0}^{\infty} \PP\left\{ \left\vert \frac{||z_{j}||_{\infty}}{b_j} -1 \right\vert > \epsilon \right\}
\end{equation}
is finite, then according to the first Borel-Cantelli lemma, it is
equivalent to $||z_{j}||_{\infty}/b_j \to 1$ almost surely. The
asymptote of the summands in (\ref{eq:asSum}) is:
\begin{eqnarray*}
&&\PP\{ \vert ||z_{j}||_{\infty}/b_j - 1 \vert > \epsilon \} = \PP\{ \vert ||z_{j}||_{\infty} -
b_j \vert > \epsilon b_j \}\\
&=& \exp\{ -e^{\epsilon b_j/a_j } \} [1+o(1)] + 1- \exp\{ -e^{-\epsilon
b_j/a_j } \} [1+o(1)]\\
&=& \exp\{ -e^{\epsilon b_j/a_j } \} [1+o(1)] + \exp\{ -\epsilon b_j/a_j \}
[1+o(1)].
\end{eqnarray*}
Series $\sum_{j=j_0}^{\infty} \exp\{ -\epsilon b_j/a_j \}$ converges if for some $\delta>0$, $\exp\{ -\epsilon b_j/a_j \}\leq C j^{-1-\delta}$ for all $\epsilon >0$, i.e. if $\log j a_j/b_j \leq \frac{\epsilon}{1+\delta}\log j$.
By the assumption of the theorem,  $$
\frac{\log j}{b_j/a_j} = \frac{\log j}{ 2^{j} \pi_j b_j H_+'(b_j)}\to 0 \,\, \text{as} \,\, j \to 0,
$$
this condition is satisfied and hence the series converges for any $\epsilon >0$. Series of
$\exp\{ -e^{\epsilon b_j/a_j } \} = o(e^{-j^{\epsilon}})$ also
converges. Therefore $\frac{ ||z_{j}||_{\infty}}{b_j} \rightarrow
1$ holds almost surely.

Hence, the necessary and sufficient condition for $f$ to belong to
$B^{s}_{p,q}$ almost surely is the finiteness of the series
$\sum_{j=1}^{\infty} 2^{js' q } \tau_j^q b_j^q$ if $1 \leqslant q
< \infty$, and in  case $q=\infty$, the finiteness of
$\sup_{j\geqslant j_0}\{2^{j s'}\tau_j b_j\}$  which is the
necessary and sufficient condition for $f \in B_{p q}^s$ almost
surely stated in the theorem.

b)  $H_{+}(x) = H(x) - H(-x) \in D(e^{-x^l})$, $x\geqslant
0$, for some $l>0$.  For this type of distributions, the
asymptotic distribution of $||z_{j}||_{\infty}$ is
\begin{eqnarray*}
\PP\{||z_{j}||_{\infty} > b_j + a_j x \} \to 1 - \exp\{ -x^{-l} \},
\end{eqnarray*}
and the constants can be chosen as  $b_j = H_+^{-1}(1-(\pi_j
2^j)^{-1})$, $a_j= g(b_j)$ for large $j$  \citep{b:led-83}. For the distributions of this type, $g(x)/x
\to 1/l$ as $x \to \infty$, and thus $a_j = b_j/l[1+o(1)]$
as $j\to \infty$. Thus,
$\PP\{||z_{j}||_{\infty} > y b_j \} \to 1 - \exp\{ -|(y-1)/l|^{-l} \}$.
 It is easy to verify that this distribution has finite moments of order $k<\ell$, i.e. $\forall k<\ell$, $\E||z_{j}||_{\infty}^k <\infty$.

Therefore the asymptotic distribution of
$\zeta_j=||z_{j}||_{\infty}/b_j$ is Type 2 of the Extreme Value
Distributions. In order to show that in case $1 \leqslant q < l$
the sum $\sum_{j=j_0}^{\infty} 2^{js' q } \left( \tau_j b_j
\zeta_j \right)^q$ converges almost surely, we apply the three
series theorem together with the theorem of monotone convergence.
Since $\zeta_j^q$ are independent non-negative random variables
with the same asymptotic distribution and positive finite mean,
the convergence of the above sum with probability 1 is equivalent to convergence of the
sum $\sum_{j=j_0}^{\infty} 2^{js' q } \tau_j^q b_j^q$.

In case $q=\infty$ we need to show that $\sup_{j\geqslant
j_0}\{ 2^{j s' }\tau_j b_j \zeta_j \} < \infty$
almost surely. The asymptotic distribution of $\max_{j_0 \leqslant j
\leqslant n}\{2^{j s' }\tau_j b_j \zeta_j \}$ as $n\to \infty$ is the Extreme Value Distribution of the
second type with its distribution function being:
\begin{eqnarray*}
\prod_{j=j_0}^n \exp\{-(2^{-j s' }\tau_j^{-1} b_j^{-1} x)^{-l}\}
= \exp\left\{-x^{-l} \sum_{j=j_0}^n 2^{j l s' }\tau_j^{l} b_j^{l}\right\}.
\end{eqnarray*}
This asymptotic distribution is proper if and only if the sum
$\sum_{j=j_0}^n 2^{j l s' }\tau_j^{l} b_j^{l}$ converges as $n
\to \infty$. Thus, by Theorem 2 of \citet{p:don-joh-98}, it is equivalent to $f\in
B_{p,q}^s$ with probability 1.


\par

\ref{it3:owt}. Now we essentially quote the argument of  \citet{p:abr-sap-silv-98} to show that in this case the distribution of $||z_j||_m^m$ is asymptotically the same for all large $j$.
Our model for wavelet coefficients assumes that with probability
$1-\pi_j$ a wavelet coefficient is zero and with probability $\pi_j$
it is non-zero, the coefficients are independent and there are $2^j$ of
them at level $j$. So the number $M_j$ of non-zero wavelet
coefficients at level $j$ is binomial with parameters $2^j$ and
$\pi_j$. For sufficiently large $j$, $M_j = N_j$ almost surely where
$N_j$ is a Poisson random variable with parameter $C_2=\lim_{j\to\infty} 2^j\pi_j$. Therefore
$N_j$ has the same distribution for every $j$. Thus if $\zeta_j$ is a
vector of $N_j$ independent random variables with distribution $H$ which is
independent of $j$,
norms of $\zeta_j$ and $z_{j}$ are equal almost surely:
\begin{eqnarray*}
|| z_j||_m^m &=& \sum_{k=0}^{2^j-1} |z_{jk}|^m = \sum_{l=1}^{M_j}
|z_{jk}^{(l)}|^m \stackrel{D}{=} \sum_{l=1}^{N_j} |z_{jk}^{(l)}|^m
\stackrel{D}{=} \sum_{l=1}^{N_j} |\zeta_{jl}|^m = ||
\zeta_j||_m^m,
\end{eqnarray*}
where $z_{jk}^{(l)}$ are non-zero normalised wavelet coefficients.
In the case $m<1$, $|| z_j||_m$, strictly speaking, is not a norm but we use the same
notation as in the case
$m\geqslant 1$ for convenience.

Therefore, by the equivalence of the norms given by equation
(\ref{eq:normequ}), $f\in B^s_{p,q}$ almost surely if and only if
\begin{eqnarray}\label{finit3}
|| w||_{b_{p,q}^s}   \stackrel{D}{=}\left[ ||u_{j_0}||_p +
\sum_{j=0}^{\infty} 2^{js'
q}\tau_j^q ||\zeta_j ||_p^q \right]^{1/q}< \infty \quad a.s.,
\end{eqnarray}
where $\zeta_j$ is a vector of $N_j$ independent identically
distributed random variables with distribution $H$, 
and $N_j$ has Poisson distribution with
parameter $C_2$. To find the condition of convergence
of the norm (\ref{finit3}) we use the property following from the
monotone convergence and the three series theorems that if $Z_n$
are independent and identically distributed non-negative random
variables with strictly positive finite mean, and $a_n$ are
non-negative constants, then $\sum a_n Z_n $ is convergent almost
surely if and only if $\sum a_n$ is convergent. By Lemma~\ref{lem:finitm}, $\E||\zeta_j||_p^q$ is finite due to assumption $\E|\xi|^q<\infty$. Therefore, condition (\ref{finit3}) is equivalent to condition $ \sum_{j=j_0}^{\infty}
2^{js'q}\tau_j^q < \infty$. Thus, by Theorem 2 of  \citet{p:don-joh-98}, it
is equivalent to $f\in B_{p,q}^s$ with probability 1.



\ref{it4:owt}. In this case $f\in B_{p,\infty}^s$ almost surely if and only if
$\sup_{j \geqslant j_0} \{ 2^{js'} \tau_j ||\zeta_j||_p \} <
\infty$, where
$\zeta_j$ is  defined in above. Appealing to
Borel-Cantelli lemmas, we deduce that the former condition holds if
and only if there exists a constant $c>0$ such that
$$
\sum_{j=j_0}^{\infty} \PP\left\{  2^{js'} \tau_j ||\zeta_j||_p\geqslant c \right\} = \sum_{j=j_0}^{\infty} [1 - (1-\pi_j)^{2^j}]\PP\left\{  2^{js'} \tau_j ||\zeta_j||_p\geqslant c \mid ||\zeta_j||_p >0 \right\} < \infty
$$  The random variables
$||\zeta_j||_p$ are independent and identically distributed  so the condition can be rewritten as
\begin{equation}\label{eq:p1}
\exists c>0: \quad \sum_{j=j_0}^{\infty} \PP\{ \eta \geqslant c  2^{-js'} \tau_j^{-1} \} < \infty,
\end{equation}
where random variable $\eta$ has the same distribution as
the conditional distribution of $||\zeta_j||_p$ given $||\zeta_j||_p >0$.
Note that the limit of $[1 - (1-\pi_j)^{2^j}]$ is $1 - e^{-C_2}$.
If $ 2^{js'} \tau_j$ is a  monotonically
decreasing function of $j$ for $j \geqslant \tilde{j}$, the series above, up to a constant, can be rewritten as
$\sum_{j=\tilde{j}}^{\infty} \PP\{ M^{-1}(\eta/c) \geqslant j \}$  where $M(j) = 2^{-js'} \tau_j^{-1}$. Since $g(x) =
M^{-1}(x/c)$ monotonically increases by the assumption of the
theorem, by Lemma 2 in  \citet[p.239]{b:fell2},  the latter condition is
equivalent to condition  $\E M^{-1}(\eta/c) < \infty$.

Therefore the $b_{p,q}^s$ norm of the wavelet coefficients is
finite if and only if $\exists c > 0$: $\E M^{-1}(|z_{jk}|/c) <
\infty$. Thus, by Theorem 2 of  \citet{p:don-joh-98}, it is equivalent to $f\in
B_{p,q}^s$ with probability 1.


\par

\ref{it5:owt}. In this case, the expected number of the non-zero wavelet coefficients is finite. Therefore, a function
with wavelet coefficients obeying (\ref{eq:prior}) is a finite
linear combination of wavelet functions with probability 1,
  belonging almost surely to the same Besov spaces as the
wavelet function, i.e. to the Besov spaces with parameters
$0<s<r$, $1\leqslant p, q \leqslant \infty$ where $r$ is the
regularity of the wavelet function.

The theorem is proved.



In the proof of the theorem  we used the following lemmas.

\begin{lemma}\label{lem:equiv}
For any $0 < v < l \leqslant \infty$, $x\in\mathbb{R}^n$, $n\in
\mathbb{N}$,
$$||x||_l  \leqslant ||x||_v \leqslant n^{\frac 1 v -
\frac 1 l} ||x||_l.$$
\end{lemma}
The proof follows directly from the H\"older inequality for finite sums.


\begin{lemma}\label{lem:finitm}
Consider the model $(\ref{eq:prior})$ with $\beta=1$, $1\leqslant
p \leqslant \infty$. Define random variables $\zeta_{jk} \sim H$,
$k=1,\dots, N_j$, where $N_j \sim Pois(C_2)$ is a Poisson random
variable, $0<C_2<\infty$. Then the following statements hold:

\begin{enumerate}


\item For $\forall m >0$ and sufficiently large $j$, $\E |\zeta_{jk}|^m  < \infty  \, \Leftrightarrow \,   \E||\zeta_j||_p^m < \infty$.

\item More generally, for any monotonic function
$g$  and sufficiently large $j$,
$$\E g(|\zeta_{jk}|) < \infty \quad
\Leftrightarrow \quad   \E g(||\zeta_j||_p) < \infty.
$$

\end{enumerate}
\end{lemma}
By $\xi_1 \stackrel{D}{=} \xi_2$ we denote the equality of
distributions of $\xi_1$ and $\xi_2$.

\begin{proof}[Proof of  Lemma \ref{lem:finitm}]

1. Since $\PP\{N_j < \infty\}=1$, for $N_j > 0$ we can apply Lemma
\ref{lem:equiv} to norms of  $\zeta_{jk}$,  implying that
$||\zeta_{j}||_l  \leqslant ||\zeta_{j}||_v \leqslant N_j^{\frac 1
v - \frac 1 l} ||\zeta_{j}||_l$. Since for $N_j=0$
$||\zeta_{j}||_v=0$ and thus the inequality is trivial, taking
$l=p$, $v=m$ for $p>m$ and
 $l=m$, $v=p$ for $p<m$, we have that
$\E |\zeta_{jk}|^m < \infty$ $\Leftrightarrow$ $\E||\zeta_j||_p^m  <
\infty$. According to the first statement of the lemma, for
sufficiently large $j$, norms $||\zeta_j||_m$ and $||z_j||_m$ have
the same distribution, therefore finiteness of the $m$th absolute
moment of $z_{jk}$ is equivalent to $\E||\zeta_j||_p^m <\infty$,
$1\leqslant p< \infty$. For $m=p$ the statement is trivial.

2. First we prove this statement for $p=\infty$ for a  monotonically increasing function $g$. On one hand,
\begin{eqnarray*}
 \E g(||\zeta_j||_{\infty})  &=& \E g\left(\max_{k=1\dots
N_j}\{|\zeta_{jk}|\}\right) = \E \max_{k=1\dots
N_j}\{g(|\zeta_{jk}|)\} I(N_j>0) \\ &\leqslant& \E \sum_{k=1}^{N_j}
g(|\zeta_{jk}|) I(N_j>0) = \E (N_j\mid N_j>0) \E
g(|\zeta_{jk}|)  < \infty,
\end{eqnarray*}
and on the other hand,
\begin{eqnarray*}
 \E g(||\zeta_j||_{\infty})  = \E g\left(\max_{k=1\dots
N_j}\{|\zeta_{jk}|\}\right) \geqslant \E [g(|\zeta_{jk}|) I(N_j>0)]
= \PP(N_j>0) \E g(|\zeta_{jk}|).
\end{eqnarray*}
If $g$ monotonically decreases, then all inequalities are reversed.

Therefore, the statement holds for $p=\infty$ and monotonic $g$. Using equivalence
of norms and that the expectation of $N_j$ is finite, we obtain this result for arbitrary $1 \leqslant p \leqslant \infty$.

\end{proof}

\section{Proofs for continuous wavelet
transform}\label{sec:ProofCont}

\subsection{Link between orthogonal and continuous wavelet
coefficients}\label{sec:KernelProp}


Since the equivalence between the Besov norm of function $f$  and
its wavelet transform is given in terms of the orthogonal wavelet
$w_{jk} = \langle\psi_{jk}, f\rangle$ and scaling  $u_{j_0 k} =
\langle\phi_{j_0 k}, f\rangle$ coefficients, in order to prove the
finiteness of the Besov norm we need to express the orthogonal
wavelet coefficients in terms of known coefficients
$\omega_{\lambda}$:
\begin{eqnarray*}
w_{jk} = \langle \psi_{jk}, f\rangle = \sum_{\lambda \in S}
\langle\psi_{jk}, \psi_{\lambda}\rangle \omega_{\lambda}+
\sum_{i=1}^{M} \eta_{\lambda_i} \langle\psi_{jk},
\phi_{\lambda_i}\rangle = \sum_{\lambda \in S} K(\lambda,
\lambda_{jk}) \omega_{\lambda},
\end{eqnarray*}
where $K(\lambda, \lambda') = \langle\psi_{\lambda},
\psi_{\lambda'}\rangle$ is the reproducing kernel of the wavelet
transform defined in Section \ref{sec:ContWavDef}, and $\psi_{jk}
= \psi_{\lambda_{jk}} = \psi_{(2^j, k 2^{-j}) }$. Note that
$\langle\psi_{\lambda}, \phi_{\lambda_i}\rangle=0$ for $\lambda
\in S$ and considered $\lambda_i$ since $\psi_{\lambda} \in
W_{j_0} \perp V_{j_0} \ni \phi_{\lambda_i}$. If we introduce
another kernel $W(\lambda, \lambda') = \langle\phi_{\lambda},
\psi_{\lambda'}\rangle$ we can write a similar expression for
$u_{j_0 k}$:
\begin{eqnarray*}
u_{j_0 k} &=& \langle\phi_{j_0 k}, f\rangle =  \sum_{\lambda \in
S} W(\lambda, \lambda_{j_0 k}) \omega_{\lambda} + \sum_{i=1}^M
\eta_{\lambda_i} \langle\phi_{\lambda_i}, \phi_{j_0 k}\rangle.
\end{eqnarray*}
The second summand is a constant we denote by $C_w$.
Therefore the orthogonal wavelet $w_{jk}$ and scaling $u_{j_0 k}$ coefficients
can be expressed in terms of  random wavelet coefficients
$\omega_{\lambda}$ in the following way:
\begin{eqnarray}\label{eq:dyadwc}
w_{jk} = \sum_{\lambda \in S} K(\lambda, \lambda_{jk}) \omega_{\lambda},\quad
u_{j_0 k} = \sum_{\lambda \in S} W(\lambda, \lambda_{j_0 k})
\omega_{\lambda} + C_w.
\end{eqnarray}


Define $K_0(u, v) = \langle\psi, \psi_{u v}\rangle$ so that
$
K(\lambda, \lambda') = K_0(a/a', a'(b-b'))$, where $\lambda = (a, b)$, $\lambda' = (a', b')$. In particular
case, where $\lambda'= \lambda_{jk}$, $K(\lambda, \lambda_{jk}) =
K_0(a 2^{-j}, 2^{j} b - k )$. Note that $K_0(u, v) = \int \psi(x) \psi(u(x-v))dx \neq 0$ only if $(0,1) \cap (v, v+1/u) \neq \emptyset$, i.e. if  $v\in (-1/u,1)$.

We will use the following estimates for the kernel of compactly
supported wavelets with $r$ vanishing moments \citep{p:abr-sap-silv-00}:
\begin{eqnarray}\label{eq:BoundK0}
|K_0(u, v)| &\leqslant& C_{K0} u^{-(r + \rho + 1/2)}, \quad u \geqslant 1; \\
|K_0(u, v)| &\leqslant& C_{K1} u^{r + \rho + 1/2}, \quad u \leqslant 1. \notag
\end{eqnarray}

Since kernel $W$ is based on the scaling function $\phi$ which has
the same regularity properties as the wavelet function $\psi$, it
can be bounded in the same way as kernel $K$.


\subsection{Proof of Theorem \ref{th:ContBesov_Gen}}\label{sec:ProofCWT}

Since we need to prove finiteness of the Besov sequence norm
$b_{p,q}^s$ of wavelet coefficients $\{w_{jk}\}$ of the orthogonal
wavelet transform which is a function of the $p$th power of the
absolute values of wavelet coefficients $w_{jk}$ we need to know
the expectation of $|w_{jk}|^p$. The following lemma  describes
the asymptotic behaviour of the absolute moments of wavelet and
scaling  coefficients.
\begin{lemma}\label{lem:moment}
Let $w_{jk}$ and $u_{j_0 k}$ be the orthogonal wavelet and scaling
coefficients corresponding to wavelet coefficients
$\omega_{\lambda}$. Suppose that the assumptions of Theorem \ref{th:ContBesov_Gen} hold,  and for some $m
\in [1,\infty)$, $q\in[1,\infty]$ and $S>0$, $\E|\xi|^m<\infty$ and $m(r+\rho+1/2)>1$.

 Then, for large $j$,
\begin{eqnarray} \label{eq:EwpA}
\E| w_{jk}|^{m} &\leq& c_{m}   [  2^{-j [m(r+\rho+1/2)-1] }  +   [\tau(2^j)]^m \mu(2^j) ],\\
\E|u_{j_0 k}|^{m} &\leq& 2c_{m} [  2^{-j_0 [m(r+\rho+1/2)-1] }  +   [\tau(2^{j_0})]^m \mu(2^{j_0}) ] + 2|C_w|^m,\notag
\end{eqnarray}
where $c_{m} = c C_K^m \nu_{m}$ for some $c>0$, $C_K = \max(C_{K0}, C_{K1})$ and $\nu_{m} = \int |x|^m dH(x)$.


\end{lemma}


\begin{proof}[Proof of Lemma \ref{lem:moment}]

The wavelet and scaling coefficients depend on two types of random variables: the Poisson process $S$ and the
wavelet coefficients (\ref{eq:dyadwc}) with indices coming from this process. In order to find the moments
of wavelet and scaling coefficients we start with the moments conditioned on the Poisson process.

1. The $m$th absolute moment of $(w_{jk}\vert S)$.

Applying the Jensen inequality, we obtain an upper bound on the
$m$th absolute moment of $w_{jk}$ conditioned on  Poisson process
$S$:
\begin{eqnarray*}
\E\left(|w_{jk}|^m \vert S\right) &=& \E\left(\left|\sum_{\lambda
\in S} K(\lambda, \lambda_{jk}) \omega_{\lambda}\right|^m \mid
S\right)
\leqslant \E\left(\sum_{\lambda \in S} |K(\lambda, \lambda_{jk})|^m |\omega_{\lambda}|^m | S \right ) \\
&=& \sum_{\lambda \in S} |K(\lambda, \lambda_{jk})|^m
\E\left(|\omega_{\lambda}|^m \vert S\right) = \nu_m \sum_{\lambda
\in S} |K(\lambda, \lambda_{jk})|^m  \tau^m(a),
\end{eqnarray*}
where $\nu_m$ is the $m$th absolute moment of distribution $H$, $\lambda = (a, b)$.
If we narrow the former sum on the following rectangles
\begin{equation}
{\cal I}^0_{jk} = [a_0, \infty) \times (2^{-j}k - 1/2, 2^{-j}k + 1/2),
\end{equation}
it does not change since $K(\lambda, \lambda_{jk}) \neq 0$  only if
$$
b-k2^{-j} \in [-1/a, 2^{-j}] \subseteq [-2^{-j_0}, 2^{-j_0}]\subseteq [-1/2, 1/2].
$$

Consider a map $f_{jk}$: $[a_0, \infty)\times [0, 1]\rightarrow
(0,+\infty)\times (-\infty, +\infty)$ which transforms the
arguments of the kernel $K$ to the arguments of the kernel $K_0$,
i.e. it transforms
  $\lambda = (a, b)$ to $(u, v) = (2^{-j}a, 2^j b - k)$.
The image of a rectangle ${\cal I}^0_{jk}$ under the transformation $f_{jk}$ is
the following rectangle:
\begin{eqnarray*}
{\cal I}_{j} := f_{jk} ({\cal I}^0_{jk}) = [a_0 2^{-j}, \infty) \times [-2^{j-1}, 2^{j-1}].
\end{eqnarray*}

Let $S'_j$ be a Poisson process on  $\Omega=\{ (u, v): u>0,
-\infty < v < +\infty\}$ with intensity $\mu(2^{j}u)$. Applied to
$S\cap {\cal I}_{jk}^0$, the map $f_{jk}$ gives us a Poisson
process with the same distribution as $S'_j\cap {\cal I}_{j}$.

Now we make a substitution $(u, v)=f_{jk}(a,b)$ in the obtained sum in order to have the expectation in terms of the kernel $K_0$:
\begin{eqnarray}\label{eq:EwpS}
\E(|w_{jk}|^m \vert S) &\leqslant& \nu_m \sum_{\lambda \in S}
|K(\lambda, \lambda_{jk})|^m  \tau^m(a)\notag\\
 &=& \nu_m  \sum_{(u, v) \in S_j^{'}\cap {\cal I}_j} |K_0(u,v)|^m  \tau^m(2^j u).
\end{eqnarray}

\vspace{0.1cm}

2. Asymptotic behaviour of $\E (|w_{jk}|^m)$.

Denote the sum in (\ref{eq:EwpS})  by $Z_j^{(m)} = \sum_{(u,v)\in
S^{'}_j\cap {\cal I}_{j}} |K_0(u,v)|^m \tau^m(2^j u)$. Then, an
upper bound for the expectation of a wavelet coefficient can be
found as $\E |w_{jk}|^m \leqslant \nu_m \E Z_j^{(m)}$.

We can use the kernel properties  \eqref{eq:BoundK0} to
find the upper bounds for the summands of $Z_j^{(m)}$:
\begin{eqnarray}\label{eq:K0est}
  |K_0(u,v)|^m \tau^m(2^j u) &\leqslant& C_{K0}^m u^{-m (r+\rho+1/2)} \tau^m(2^j u) \quad \text{for} \quad u\geqslant 1;\\
 |K_0(u,v)|^m \tau^m(2^j u) &\leqslant& C_{K1}^m u^{m (r+\rho+1/2) } \tau^m(2^j u) \quad \text{for} \quad u\leqslant 1. \notag
\end{eqnarray}
Now, to estimate the expectation of $Z_j^{(m)}$, we apply the following version of
Corollary 1 by \citet{p:abr-sap-silv-00}, with $\varepsilon \mu$ being replaced by $\mu_{\varepsilon}$ (the proof of Corollary 1 and Lemma 1 in \citet{p:abr-sap-silv-00} goes through exactly with this replacement under the assumption that $c_{l, \varepsilon} \to 0$).

\begin{corollary}[Corollary 1 in \citet{p:abr-sap-silv-00}]\label{cor:ASS1}
Let $\mu_{\varepsilon}$ be a measure on a set $\Omega$, and let $g$ be a real-valued function on $\Omega$.
Let $S_{\varepsilon}$ be a Poisson process on $\Omega$ with intensity $ \mu_{\varepsilon}$, where $\varepsilon>0$.
Assume that the induced measure $\mu_{\varepsilon}(g_{\varepsilon}^{-1}(A))$ is non-atomic for every measurable set $A  \subseteq\mathbb R$, and assume that for some $l>0$,
\begin{equation*}
\int_{\Omega} \min(1, |g(x)|) \mu_{\varepsilon}(dx) < \infty, \quad c_{l, \varepsilon} =
\int_{\Omega}  |g_{\varepsilon}(x)|^l \mu_{\varepsilon}(dx) \to 0 \, \text{as} \,\,\varepsilon \to 0.
\end{equation*}
Define $Y_{\varepsilon} = \sum_{X\in S_{\varepsilon}} g_{\varepsilon}(X)$. Then
$\E |Y_{\varepsilon}|^l = c_{l, \varepsilon}[1 +o(1)]$  as
 $\varepsilon\rightarrow 0$.
\end{corollary}

We  apply the corollary with
 $X = (u,v)$, $\Omega = [0, \infty)\times (-\infty, +\infty)$,
 $d\mu_{\varepsilon}(u,v) = \mu(2^j u) du dv$, $\varepsilon = 2^{- j}$ and
$g_{\varepsilon}(X) = |K_0(u,v)|^m [\tau(2^{j}u)]^m \geqslant 0$ is a continuous function.
The assumptions of the corollary are satisfied if the integral $\int_{\Omega}  g_{\varepsilon}(x) d\mu_{\varepsilon}(x)$ tends to 0 for large $j$.
Using the bound (\ref{eq:K0est}) on $|K_0(u,v)|^m [\tau(2^j u)]^m$, we can estimate
the required integral:
\begin{eqnarray*}
c_{j} &=& \int_{\Omega}g_{2^{-j}}(x) d\mu_{2^{-j}}(x)  = \int_0^{+\infty}\int_{-\infty}^{+\infty} |K_0(u,v)|^m [\tau(2^j u)]^m du dv\\
&=&  \int_0^{1}\int_{-\infty}^{+\infty} |K_0(u,v)|^m [\tau(2^j u)]^m \mu(2^j u) du dv\\
&  +& \int_1^{+\infty}\int_{-\infty}^{+\infty} |K_0(u,v)|^m [\tau(2^j u)]^m \mu(2^j u)du dv\\
 &\leqslant&  C_{K0}^m \int_0^{1}  u^{m(r+\rho+1/2)}(1+1/u) [\tau(2^j u)]^m \mu(2^j u) du\\
&+& C_{K1}^m \int_1^{+\infty} u^{- m(r+\rho+1/2) }(1+1/u) [\tau(2^j u)]^m \mu(2^j u) du,
\end{eqnarray*}
since the kernel properties discussed in  Appendix \ref{sec:KernelProp} imply $v\in [- 1/u, 1]$.

Consider the second integral. Making the change of variables $z = 2^j u$ and using the assumption that functions $\tau(a)$ and $\mu(a)$ are monotonically decreasing, the integral is bounded by
\begin{eqnarray*}
&&C_{K1}^m 2^{j[m(r+\rho+1/2)-1] } \int_{2^j}^{+\infty}  z^{- m(r+\rho+1/2) }[\tau(z)]^m \mu(z) dz\\
&\leq& C_{K1}^m 2^{j[m(r+\rho+1/2)-1] }  \sum_{k=j}^{\infty}   2^{-k m(r+\rho+1/2)+k}[\tau(2^{k})]^m \mu(2^{ k})\\
 &\leq& C_{K1}^m  [\tau(2^j)]^m  \mu(2^j) \sum_{k=0}^{\infty}   2^{-k [m(r+\rho+1/2)-1]}= C C_{K1}^m  [\tau(2^j)]^m  \mu(2^j)
 \end{eqnarray*}
which tends to zero due to assumption $m(r+\rho+1/2)-1 >0$ and that function $\tau(a)$ or $\mu(a)$ monotonically decreases to 0.

Now consider the first integral. By the assumptions of the lemma, function $  \tau(z)  [\mu(z)]^{1/m}$ decreases for $z\geq a_0$ hence this function
 either decreases on $z\geq a_0$, or it has a maximum at some point $z_0 > a_0$ that is independent of $j$ which implies that for all $z\geq a_0$,
 $$
\int_{a_0}^{2^j} z^{ m (r+\rho+1/2)-1}[\tau(z)]^m \mu(z)dz \leq C[1  + 2^{ j m (r+\rho+1/2)-1}[\tau( 2^j)]^m \mu( 2^j)] 2^j
 $$
for some constant $C>0$.
Then,  \begin{eqnarray*}
&& \int_0^{1}  u^{m(r+\rho+1/2)-1}[\tau(2^j u)]^m \mu(2^j u) du =
2^{-j }\int_0^{2^j}  (2^{-j}z)^{ m (r+\rho+1/2)-1}[\tau(z)]^m \mu(z) dz\\
&\leq& 2^{-j m(r+\rho+1/2) }\int_0^{a_0} z^{ m (r+\rho+1/2)-1}[\tau(z)]^m \mu(z) dz\\
 &+& C 2^{-j m(r+\rho+1/2) }\int_{a_0}^{2^j} z^{ m (r+\rho+1/2)-1}[\tau(z)]^m \mu(z)  dz\\
&\leq& C 2^{-j [m(r+\rho+1/2) -1]} + C  [\tau( 2^j)]^m \mu( 2^j) \to 0
\end{eqnarray*}
as $j\to \infty$ since $m(r+\rho+1/2)-1 >0$ by the assumptions of the lemma.

Hence both assumptions of Corollary \ref{cor:ASS1} are satisfied implying that
$$
\E Z_j^{(m)} \leq C C_K^m [  2^{-j [m(r+\rho+1/2)-1] }+[\tau(2^j)]^m \mu(2^j)] 
$$
for large enough $j$ which gives a bound for the $m$th absolute moment of $w_{jk}$:
\begin{eqnarray}
\E|w_{jk}|^m \leq  \nu_m c C_K^m [  2^{-j [m(r+\rho+1/2)-1] }+   [\tau(2^j)]^m \mu(2^j) ]
\end{eqnarray}
for some constant $c>0$ which proves the lemma.

Since the kernels $W$ and $W_0$ defined in Appendix \ref{sec:KernelProp} have the same properties as the properties of
the kernels $K$ and $K_0$ we used in the proof, the same asymptotic result holds for $ \E|u_{j_0 k}|^m$.

\end{proof}

Now we prove  Theorem \ref{th:ContBesov_Gen}. Throughout the proof we use the following statement: $\E X <
\infty$ and $X>0$ imply $\PP\{ X < \infty\}=1$ which is due to Markov inequality.

\begin{proof}[Proof of Theorem \ref{th:ContBesov_Gen}]
{\it Sufficiency. }

\ref{it1:cwt}.  We need to prove that if the series converges then $|| w ||_{b_{p,q}^s}<\infty$ a.s. Conditions of  Lemma~\ref{lem:moment} are satisfied with $m=p$ and $S = s+1/2$.
Applying Lemma~\ref{lem:moment}, we have $||w||_{b^s_{p,q}}<\infty$ with probability 1 since its expectation is finite:
\begin{eqnarray*}
\E ||w||_{b^s_{p,q}} &=&  \E \left[\sum_{j=j_0}^{\infty} 2^{j q s'}  \left( \sum_{k=0}^{2^j-1} |w_{jk}|^p \right)^{q/p}\right]^{1/q} + \E \left( \sum_{k=0}^{2^{j_0}-1} |u_{j_0 k}|^p \right)^{1/p}\\
 &\leqslant&   \left[ C_{p,q}^q\sum_{j=j_0}^{\infty} 2^{j q s'} \left(  \sum_{k=0}^{2^j-1} \E |w_{jk}|^p \right)^{q/p}\right]^{1/q}+
\left( \sum_{k=0}^{2^{j_0}-1} \E |u_{j_0 k}|^p \right)^{1/p}\\
 &\leqslant&  C \left[ C_{p,q}^q\sum_{j=j_0}^{\infty} 2^{j q (s+1/2-1/p)} \left(2^j 2^{-j p(r+\rho+1/2) }  +  2^j 2^{ -j p (s+1/2 +1/q)} \right)^{q/p}\right]^{1/q}\\
&&+ 2^{j_0/p} c_p^{1/p} c_{j_0}^{1/p} + 2^{j_0/p} |C_w|\\
 &\leq& C C_{p,q} c_p^{1/p}\left[ \sum_{j=j_0}^{\infty}  [ 2^{-j q(r+\rho-s)  }+  2^{-j}   ] \right]^{1/q} + 2^{j_0/p} c_p^{1/p} c_{j_0}^{1/p} + 2^{j_0/p} |C_w|
\end{eqnarray*}
under the assumptions of the theorem as $r+\rho-s >\rho> 0$. 

\ref{it2:cwt}.  Similarly to the proof of case \ref{it1:cwt},
 we apply Lemma~\ref{lem:moment} with $m=p$ and $S = s+1/2$:
\begin{eqnarray*}
\E ||w||_{b^s_{p,\infty}} &=&  \E \sup_{j \geqslant j_0} 2^{js'}  \left( \sum_{k=0}^{2^j-1} |w_{jk}|^p \right)^{1/p} + \E \left( \sum_{k=0}^{2^{j_0}-1} |u_{j_0 k}|^p \right)^{1/p}\\
&\leqslant&  \sup_{j \geqslant j_0} 2^{js'} \left(  \sum_{k=0}^{2^j-1} \E |w_{jk}|^p \right)^{1/p}+
\left( \sum_{k=0}^{2^{j_0}-1} \E|u_{j_0 k}|^p \right)^{1/p}\\
 &\leqslant& c_p^{1/p}C\sup_{j\geqslant j_0} 2^{j s'}[ 2^{j/p} 2^{-j (r+\rho+1/2) }  +  2^{j/p} 2^{ -j  (s+1/2)} ]+ c_p^{1/p} c_{j_0} + 2^{j_0/p} |C_w|\\
  &=& c_p^{1/p}C\sup_{j\geqslant j_0} [   2^{-j (r+\rho-s) }  +   1 ]+ c_p^{1/p} c_{j_0} + 2^{j_0/p} |C_w|
\end{eqnarray*}
which is finite under the conditions of the theorem.
Therefore the sufficiency is proved.

{\it Necessity.}


The idea is to reduce the problem to the known case considered in Section \ref{sec:regularDWT}, i.e. to the model where
coefficients $w_{jk}$ are independent and whose distribution is a mixture of the point mass at zero and a scaled distribution $H$.
In the considered case $w_{jk}$ is a sum of the same independent random variables with different coefficients, therefore
$w_{jk}$ are  dependent:
\begin{equation}
w_{jk} = \sum_{\lambda \in S} K(\lambda, \lambda_{jk}) \omega_{\lambda}.
\end{equation}
To create independent wavelet coefficients, we divide the area $[a_0, \infty)\times [0,1]$ into non-overlapping rectangles
${\cal I}_{jk}$ and define new random variables $w_{jk}'$ as a sum  similar to the   sum defining $w_{jk}$  but over a restricted
summation area  $S\cap {\cal I}_{jk}$ instead of $S$:
$$
w_{jk}' = \sum_{\lambda \in S\cap {\cal I}_{jk}} K(\lambda, \lambda_{jk}) \omega_{\lambda}.
$$
Note that stochastically $|w_{jk}'|\leqslant |w_{jk}|$. Now we
define rectangles ${\cal I}_{jk}$. Since function $K_0(u, v)$ is
continuous and $K_0(1,0) =1$ we choose $c_0$ in $(0, 1)$ such that
$K_0(u,v)> 1/2$ for all $(u, v)$ in $[1, 1+c_0]\times [0, c_0] $
and define the rectangles as ${\cal I}_{jk} = [2^j, 2^j
(1+c_0)]\times [2^{-j}k , 2^{-j}(k+c_0)]$. Note that the
definition of $c_0$ implies that for any $j\geqslant j_0$,
$k=0,\dots, 2^j-1$, $K(\lambda, \lambda_{jk})> 1/2$ for $\lambda \in
{\cal I}_{jk}$.

Now we find  analogues of the proportion of non-zero coefficients
at level $j$ $\pi_j$ and the variance of a wavelet coefficient  at
level $j$ $\tau_j$ used in our model in Section
\ref{sec:mainBesov}. For sufficiently large $j$, the expected
number of wavelet indices $\lambda$ falling into the region ${\cal
I}_{jk}$ is $\int_{{\cal I}_{jk}} \mu(a) da db \leqslant c_0^2
\max\{\mu(2^j), \mu(2^j(1+c_0))\} = c_0^2 \mu(2^j)$ since $\mu(a)$
decreases for large $a$, therefore the probability that there is
at least one $\lambda$ in ${\cal I}_{jk}$  is $ c_1 \mu(2^j)$ for
some constant $c_1$.

Next we estimate the variance of $w_{jk}'$:
$$
\Var(\omega_{\lambda} \vert \lambda \in S\cap {\cal I}_{jk})
=\tau^2(a) \geqslant \tau^2(a(1+c_0)) \geqslant \tau^2(2a)
\geqslant 4 c_2 \tau^2(a)
$$
by Assumption B, so $\Var(K(\lambda, \lambda_{jk})\omega_{\lambda}
\vert \lambda \in S\cap {\cal I}_{jk})  \geqslant c_2 \tau^2(a)$
since $K(\lambda, \lambda_{jk})>1/2$ for $\lambda\in {\cal
I}_{jk}$. Therefore $\Var(w_{jk}')\geqslant c_3 \tau^2(a)$.

We define $w^0_{jk}$ for the same $j$ and $k$ as $w_{jk}$ such
that they are independent and have a mixture distribution
$$
w^0_{jk} \sim \pi_j h_j(x) + (1-\pi_j) \delta_0(x),
$$
where distribution $h_j$ has variance $\tau_j^2 = c_3 [\tau(2^{j})]^2$
and $\pi_j = \min\{1, c_1 \mu(2^{j})\}$.


By the construction of the wavelet coefficients $ w'_{jk}$ and $ w^{0}_{jk}$ the following inequalities hold
stochastically: $|w_{jk}| \geqslant |w'_{jk}| \geqslant |w^{0}_{jk} |$.
Therefore using Theorem \ref{th:mainBesov} which can be applied under the assumptions of the theorem and the proposition,
it follows that if the original function $f \in B^s_{p,q}$, i.e that
 $||w^0 ||_{b^s_{p,q}} \leqslant  ||w||_{b^s_{p,q}} <\infty$ almost surely, then $\delta<0$.
Thus the theorem and the proposition are proved.

\end{proof}


\subsection{Proof of sufficiency for $p=\infty$ and $\beta<1$ in Theorem~\ref{th:ContBesov}}\label{sec:ProofCWT_prop}

\begin{proof}

The necessity follows the proof of necessity in Theorem~\ref{th:ContBesov_Gen} exactly
since the necessary and sufficient condition in case $p=\infty$ and $\beta<1$ is the same as for the orthogonal expansion.

Now we prove the sufficiency of the statement  of Theorem~\ref{th:ContBesov} for $p=\infty$ and $\beta<1$.
The condition of the theorem in this case can be unified as $\tilde\delta<0$, where
 $\tilde\delta = s' - \alpha/2 + \delta_H$.
In the case of $H$ having a power exponential tail, $\tilde\delta$ coincides with $\delta = s' - \alpha/2 $ for $p=\infty$.

We need to show that if $\tilde\delta < 0$ then $||w||_{b_{\infty,q}^s}<\infty$, i.e.
that the following inequalities hold almost surely:
\begin{eqnarray}\label{cond:bes}
\max_{k=0 \dots 2^{j_0} -1} (|u_{j_0 k}|) + \sum_{j=j_0}^{\infty} 2^{s'qj} \{\max_{k=0 \dots 2^j-1} (|w_{jk}|) \}^q < \infty, \quad q< \infty, \\
\max_{k=0 \dots 2^{j_0} -1} (|u_{j_0 k}|) + \sup_{j\geqslant j_0}\{ 2^{js'} \max_{k=0 \dots 2^j-1} (|w_{jk}|)  \} <\infty, \quad q= \infty. \notag
\end{eqnarray}

For the proof we apply the Markov inequality that for
 a positive random variable $\xi$ and a constant $B>0$, $\PP\{ \xi > B \} \leqslant \E \xi / B$.
We consider probabilities $\PP\{ |w_{jk}| > A\}$ and $\PP\{ |u_{j_0 k}| > B\}$ for some $A, B > 0$. According to the Markov inequality,
\begin{eqnarray*}
\PP\{ |w_{jk}| >A\} &=& \PP\{ |w_{jk}|^\varkappa >A^\varkappa\} \leqslant A^{-\varkappa} \E |w_{jk}|^\varkappa,\\
\PP\{ |u_{j_0 k}| >B\} &\leqslant& B^{-\varkappa} \E |u_{j_0 k}|^\varkappa
\end{eqnarray*}
for any $\varkappa > 0$, and in the case of a distribution with polynomial tail there is an additional condition that  $\varkappa < l$.
We can use an upper bound for the expectation of the absolute moments of wavelet and scaling coefficients  given in Lemma \ref{lem:moment}:
\begin{eqnarray*}
\PP\{ |w_{jk}| >A\} &\leqslant& A^{-\varkappa} \E |w_{jk}|^{\varkappa} \leqslant c A^{-\varkappa}[ 2^{-j\varkappa(r+\rho+1/2)+j}+ 2^{-j(\beta+\varkappa\alpha/2)}];\\
\PP\{ |u_{j_0 k}| >B\} &\leqslant& c B^{-\varkappa}[ 2^{-j_0(\beta + \varkappa \alpha / 2 )} +  2^{-j_0\varkappa(r+\rho+1/2)+j_0} + c_{j_0} |C_w|]\leq C_u,
\end{eqnarray*}
as long as $\varkappa(r+\rho+1/2)>1$.
If we set $A = 2^{\gamma j}$ and $B=2^{\iota j_0}$ then the inequalities for the wavelet and scaling coefficients can be rewritten as
\begin{eqnarray*}
\PP\left\{ \max_{k=0\dots 2^j-1} |w_{jk}|>A \right\} &\leqslant& \sum_{k=0}^{2^j-1} \PP\{| w_{jk}| >A\} \leqslant  c 2^{j(-\gamma \varkappa +1 -\beta -\varkappa\alpha/2)}+ c 2^{-j\varkappa(r+\rho+1/2)+2j},\\
\PP\left\{ \max_{k=0\dots 2^{j_0}-1} |u_{j_0 k}|>B \right\} &\leqslant& c 2^{j_0(-\varkappa \iota+1 -\beta -\varkappa\alpha/2)}+ c2^{-j_0\varkappa(r+\rho+1/2)+2j_0} \leq C_u 2^{c j_0}.
\end{eqnarray*}

Therefore $\max_{k=0,\dots, 2^j-1} \{|w_{jk}|\} \leqslant A= 2^{\gamma j}$ for all $j$ almost surely if
\begin{equation}\label{eq:PrCond1}
\gamma -(1 -\beta)/\varkappa + \alpha/2  > 0\quad and \quad \varkappa(r+\rho+1/2)>2.
\end{equation}
A similar statement holds for the scaling coefficients,
i.e. $\max_{k} \{|u_{j_0 k}|\} \leqslant B=2^{\iota j_0}$ almost surely if
\begin{equation}\label{eq:PrCond2}
-\iota +(1 -\beta)/\varkappa - \alpha/2 < \infty,
\end{equation}
i.e. if $\iota >-\infty$ and $\varkappa \neq 0$.

Suppose that inequalities (\ref{eq:PrCond1}) and (\ref{eq:PrCond2})
hold. Then  conditions (\ref{cond:bes}) are satisfied if
\begin{eqnarray*}
&&\max_{k=0 \dots 2^{j_0} -1} (|u_{j_0 k}|) +  \sum_{j=j_0}^{\infty} 2^{s'qj} ||w_{j}||_{\infty}^q\\
&\leqslant& C_u 2^{j_0} + \sum_{j\geqslant j_0} 2^{qj(s' + \gamma)}+ \sum_{j\geqslant j_0} c 2^{-j\varkappa(r+\rho+1/2)+2j} <\infty
\end{eqnarray*}
in the case $q<\infty$, and if
\begin{eqnarray*}
&&\max_{k=0 \dots 2^{j_0} -1} (|u_{j_0 k}|) +\sup_{j\geqslant j_0}\{ 2^{js'} ||w_{j}||_{\infty} \}\\
 &\leqslant& C_u 2^{j_0} + c\sup_{j\geqslant j_0}\{ 2^{j(s' + \gamma) +c 2^{-j\varkappa(r+\rho+1/2)+2j}} \}<\infty
\end{eqnarray*}
in the case $q=\infty$ which both hold if $s'+\gamma < 0$.

To satisfy condition (\ref{eq:PrCond2}), we  take
  $\iota =(1 -\beta)/\varkappa + \alpha/2+1$ and choose $\varkappa\neq 0$ later.

Combining the requirement $s'+\gamma < 0$ with condition (\ref{eq:PrCond1}) we have that $\gamma$ should belong to the interval
$( (1 -\beta)/\varkappa - \alpha/2, -s')$. If there exists $\varkappa>2/(r+\rho+1/2)$ such that this interval is not empty then
the necessity is proved.
In case $\beta = 1$, the interval is not empty if and only if $\tilde \delta = s' - \alpha/2 <0$.

For $\beta < 1$, we consider the cases of distribution $H$ having polynomial or power exponential tail separately.

1. For a distribution with the power exponential tail $\tilde\delta = \delta = s' - \alpha/2$ and
the interval for $\gamma$ is $(\delta - s' + (1 -\beta)/\varkappa, -s')$ for some $\varkappa > 0$.
Since $\beta<1$ and $\delta < 0$ this interval is not empty if $\delta + (1 -\beta)/\varkappa<0$ i.e. if
 $\varkappa >-(1 -\beta)/\delta$.
For example, we can choose $\varkappa = \max\{2\frac{1 -\beta}{-\delta}, 3/(r+\rho+1/2)\}$ and $\gamma=-s' + \delta/4 \in (-s' + \delta+ (1 -\beta)/\varkappa, -s')$ since
$$ -s' + \delta/4 \geq -s' + \delta + \min(-\delta/2, (1-\beta)(r+\rho+1/2)/3) = -s' + \delta+ (1 -\beta)/\varkappa.$$

2. For a distribution with polynomial tail
 $\tilde \delta= s'-\alpha/2 +(1-\beta)/l$, and the conditions on $\varkappa$ and $\gamma$ are:
\begin{eqnarray*}
-s' + \tilde \delta + (1 -\beta)(\frac 1 {\varkappa} - \frac 1 l)
< \gamma < -s'  \quad \textsl{and} \quad  2/(r+\rho+1/2) < \varkappa < l.
\end{eqnarray*}
Thus, we must have $l >2/(r+\rho+1/2)$.

Since $\beta<1$, the interval for $\gamma$ is not empty if $\tilde
\delta + (1 -\beta)(\frac 1 {\varkappa} - \frac 1 l) < 0$ which
takes place for $\varkappa$ such that
$$
\frac 1 l < \frac 1 {\varkappa} < \min\left\{\frac 1 l +
\frac{-\tilde\delta}{1-\beta}, (r+\rho+1/2)/2\right\}.
$$
Thus there exists a pair $(\gamma, \varkappa)$ satisfying the
stated conditions. For instance, we can take $\varkappa$ and $\gamma$ such that
\begin{eqnarray*}
\frac{1}{\varkappa} &=& \frac 1 \ell + \frac 1 2 \left[\frac{-\tilde\delta}{1-\beta}, \frac{ r+\rho+1/2} 2 - \frac 1 \ell \right],\\
 \gamma &=& -s' + \tilde\delta+ \frac 1 2 \min\left\{ -\tilde\delta, (1-\beta) \left[ \frac{ r+\rho+1/2} 2 - \frac 1 \ell  \right]\right\}.
\end{eqnarray*}

Thus, if $\tilde\delta<0$, $|| w||_{b_{p, q}^s}$ is finite almost surely. Therefore the sufficiency, and hence the statement of the theorem in the considered  case, is proved.

\end{proof}

\section*{Acknowledgements}

This work has started when the author was a PhD student at the
University of Bristol, UK where she was supported by Overseas
Students Research Award Scheme. The author would like to give special thanks to Bernard Silverman for fruitful discussions,
support and encouragement during the PhD, and Theofanis Sapatinas for helpful
comments.


\end{document}